\documentclass[12pt,reqno,a4paper]{amsart}
\usepackage{fullpage}
\usepackage[utf8]{inputenc}
\usepackage[T1]{fontenc}
\usepackage{hyperref} 
\usepackage{enumitem}
\usepackage{amsmath,amssymb,amsthm,color}
\usepackage{todonotes,framed}
\usepackage{nameref}
\usepackage{listings}
\usepackage{multirow}
\usepackage[binary-units=true]{siunitx}

\usepackage{pgfplots}
\pgfplotsset{compat=1.9}
\usepackage{pgfplotstable}

\def\OO{\mathcal{O}}

\def\alg#1{{\it\bfseries{#1}}}

\usepackage{mdframed}

\def\set#1#2{\big\{#1\colon#2\big\}}
\def\prod#1#2{{\langle #1,#2\rangle}}

\def\R{\mathbb{R}}

\def\0{\boldsymbol{0}}

\def\SS{\mathcal S}
\def\TT{\mathcal T}

\def\BB{\boldsymbol{B}}

\def\MM{\boldsymbol{M}}
\def\RR{\boldsymbol{R}}

\def\II{\boldsymbol{\mathcal{I}}}

\def\N{\mathbb{N}}
\def\R{\mathbb{R}}

\def\l0r{\!\left(0\right)}

\def\Cmh0{{\tilde{C}}_{{0}}}

\newcounter{statement}
\newenvironment{statement}[2][!]{%
\vskip3mm
\hrule
\hrule
\hrule
\vskip1mm
\noindent%
\refstepcounter{statement}%
\bf#2~\thestatement%
\ifthenelse{\equal{#1}{!}}{.\ }{~(#1).\ }%
\it%
}{%
\vskip1mm
\hrule
\hrule
\hrule
\vskip2mm
}

\newenvironment{theorem}[1][!]{\begin{statement}[#1]{Theorem}}{\end{statement}}
\newenvironment{lemma}[1][!]{\begin{statement}[#1]{Lemma}}{\end{statement}}
\newenvironment{definition}[1][!]{\begin{statement}[#1]{Definition}}{\end{statement}}
\newenvironment{proposition}[1][!]{\begin{statement}[#1]{Proposition}}{\end{statement}}

\newenvironment{remark}[1][!]{\begin{statement}[#1]{Remark}}{\end{statement}}
\newenvironment{algorithm}[1][!]{\begin{statement}[#1]{Algorithm}}{\end{statement}}

\mdfdefinestyle{theoremstyle}{
} 

\def\new{{\rm new}}
\def\old{{\rm old}}
\def\mod{{\rm mod}}
\def\id{{\rm id}}

\usetikzlibrary{calc}

\title{D\"orfler marking with minimal cardinality\\ is a linear complexity problem}

\author{Carl-Martin Pfeiler and Dirk Praetorius}

\address{TU Wien, Institute for Analysis and Scientific Computing, Wiedner Hauptstr. 8-10/E101/4, 1040 Vienna, Austria}
\email{carl-martin.pfeiler@asc.tuwien.ac.at \quad \rm (corresponding author)}
\email{dirk.praetorius@asc.tuwien.ac.at}

\thanks{{\bf Acknowledgement.} The authors thankfully acknowledge support by the Austrian Science Fund (FWF) through the doctoral school \emph{Dissipation and dispersion in nonlinear PDEs} (grant W1245) and through the research project \emph{Optimal adaptivity for BEM and FEM-BEM coupling} (grant P27005).}

\makeatother

\makeatletter
\def\@seccntformat#1{\hspace*{4mm}%
  \protect\textup{\protect\@secnumfont
    \ifnum\pdfstrcmp{subsection}{#1}=0 \bfseries\fi
    \csname the#1\endcsname
    \protect\@secnumpunct
  }%
}
\makeatother

\usepackage{fancyhdr}
\advance\footskip0.4cm
\advance\textheight-0.4cm
\calclayout
\newcommand*\patchAmsMathEnvironmentForLineno[1]{%
  \expandafter\let\csname old#1\expandafter\endcsname\csname #1\endcsname
  \expandafter\let\csname oldend#1\expandafter\endcsname\csname end#1\endcsname
  \renewenvironment{#1}%
     {\linenomath\csname old#1\endcsname}%
     {\csname oldend#1\endcsname\endlinenomath}}%
\newcommand*\patchBothAmsMathEnvironmentsForLineno[1]{%
  \patchAmsMathEnvironmentForLineno{#1}%
  \patchAmsMathEnvironmentForLineno{#1*}}%
\AtBeginDocument{%
\patchBothAmsMathEnvironmentsForLineno{equation}%
\patchBothAmsMathEnvironmentsForLineno{align}%
\patchBothAmsMathEnvironmentsForLineno{flalign}%
\patchBothAmsMathEnvironmentsForLineno{alignat}%
\patchBothAmsMathEnvironmentsForLineno{gather}%
\patchBothAmsMathEnvironmentsForLineno{multline}%
}
\usepackage[mathlines]{lineno}

\def\PP{\mathfrak P}
\def\MM{\mathcal M}
\def\MMM{\mathfrak M}
\def\II{\mathcal I}
\def\JJ{\mathcal J}
\def\Nmin{N_{\rm min}}

\def\BB{\mathcal{B}}
\def\RR{\mathcal{R}}

\begin{document}

\begin{abstract}
Most adaptive finite element strategies employ the D\"orfler marking strategy to single out certain elements $\MM \subseteq \TT$ of a triangulation $\TT$ for refinement. In the literature, different algorithms have been proposed to construct $\MM$, where usually two goals compete: On the one hand, $\MM$ should contain a minimal number of elements. On the other hand, one aims for linear costs with respect to the cardinality of $\TT$. Unlike expected in the literature, we formulate and analyze an algorithm, which constructs a minimal set $\MM$ at linear costs. Throughout, pseudocodes are given.
\end{abstract}

\maketitle
\thispagestyle{fancy}

\newenvironment{explain}{\begin{list}{\LARGE\lightblue$\bullet$}{%
\setlength{\labelsep}{2.3mm}%
\setlength{\labelwidth}{22mm}%
\setlength{\leftmargin}{22mm}%
\setlength{\itemsep}{1mm}%
}}{\end{list}}

\section{Introduction}

\noindent
In the last decade, the mathematical understanding of adaptive finite element methods (AFEM) has matured. For many elliptic model problems, one can mathematically prove that AFEM leads to optimal convergence behavior; see, e.g., \cite{doerfler1996,mns00,bdd04,stevenson2007,ckns2008} for some of the seminal works for symmetric problems,~\cite{mn2005,cn2012,ffp2014} for the extension to nonsymmetric problems, or to~\cite{cfpp2014} for some recent review on the state of the art.

Starting from an initial mesh $\TT_0$, the usual AFEM algorithms iterate the loop
\begin{align}\label{eq:semr}
 \boxed{\tt solve} 
 \quad \to \quad 
 \boxed{\tt estimate} 
 \quad \to \quad 
 \boxed{\tt mark} 
 \quad \to \quad 
 \boxed{\tt refine} 
\end{align}
The latter generates a sequence $(\TT_\ell)_{\ell \in \N_0}$ of successively refined meshes together with the associated FEM solutions $u_\ell$ and {\sl a~posteriori} error estimators $\eta_\ell = [ \sum_{T \in \TT_\ell} \eta_\ell(T)^2 ]^{1/2}$, where the index $\ell$ is the step counter of the adaptive loop. Formally, the algorithm reads as follows: For all $\ell = 0,1,2, \dots$, iterate the following steps:
\begin{explain}
 \item[$\boxed{\tt solve}$] Compute the FEM solution $u_\ell$ corresponding to $\TT_\ell$.
 \item[$\boxed{\tt estimate}$] Compute certain refinement indicators $\eta_\ell(T)$ for all $T \in \TT_\ell$.
 \item[\tt $\boxed{\tt mark}$] Determine a subset of elements $\MM_\ell \subseteq \TT_\ell$ for refinement.
 \item[\tt $\boxed{\tt refine}$] Generate a new mesh $\TT_{\ell+1}$ by refinement of (at least) all marked elements.
\end{explain}
Usually, the set $\MM_\ell$ from $\boxed{\tt mark}$ then contains the elements with the largest contributions $\eta_\ell(T)$. Often (and, in particular, for the analysis of rate optimality~\cite{cfpp2014}), the D\"orfler marking criterion~\cite{doerfler1996} is used: Given $0 < \theta \le 1$, construct $\MM_\ell \subseteq \TT_\ell$ such that
\begin{align}\label{eq:Doerfler}
 \theta \eta_\ell^2 \le \sum_{T \in \MM_\ell} \eta_\ell(T)^2,
\end{align}
i.e., the marked elements control a fixed percentage $\theta$ of the overall error estimator. Clearly, one aims to choose the set $\MM_\ell$ with as few elements as possible.

As far as convergence of AFEM is concerned, also other marking criteria can be considered~\cite{msv2008,siebert2011}.
Current proofs of rate optimality of AFEM, however, rely on the (quasi-) minimal D\"orfler marking~\eqref{eq:Doerfler}, where the set $\MM_\ell$ has to be chosen with minimal cardinality (at least up to some $\ell$-independent generic constant); see~\cite{cfpp2014}. Moreover, when the focus comes to the overall computational cost of AFEM,
it is important that all steps of the adaptive algorithm can be performed at linear cost with respect to the number of elements $\#\TT_\ell$. This is usually a reasonable assumption if $\boxed{\tt solve}$ employs iterative solvers like PCG~\cite{fhps2019} or multigrid~\cite{stevenson2007}, and it requires appropriate data structures for $\boxed{\tt estimate}$ and $\boxed{\tt refine}$.

If $\boxed{\tt mark}$ aims for a set $\MM_\ell$, which satisfies~\eqref{eq:doerfler} with minimal cardinality, then linear cost is less obvious: The work~\cite{doerfler1996} notes that a possible strategy is to sort the indicators, which, however, results in log-linear costs.
Instead, the work~\cite{stevenson2007} employs an approximate sorting by binning. While this leads to linear costs, the resulting set $\MM_\ell$ has only minimal cardinality up to a multiplicative factor $2$, and \cite[Section~5]{stevenson2007} notes:
\begin{quote}
\emph{Selecting $\MM_\ell$ that satisfies~\eqref{eq:Doerfler} with true minimal cardinality would require sorting all $T \in \TT_\ell$ by the values of $\eta_\ell(T)$, which takes $\OO(N\log N)$ operations.}
\end{quote}
The present work bridges the approaches of~\cite{doerfler1996,stevenson2007} and proves that the latter statement is wrong: Based on ideas of the (Quick-) Selection algorithm~\cite{hoare1961_find}, we present a linear-cost algorithm for $\boxed{\tt mark}$, which provides a set $\MM_\ell \subseteq \TT_\ell$, which satisfies the D\"orfler criterion~\eqref{eq:Doerfler} with minimal cardinality. 

The outline of the present work reads as follows: In Section~\ref{section:doerfler}, we formulate the D\"orfler marking and briefly discuss the algorithms from~\cite{doerfler1996,stevenson2007}. 
In Section~\ref{section:main}, we present and analyze our new approach for $\boxed{\tt mark}$ named \alg{QuickMark}.
Section~\ref{section:implementation} concludes with a \texttt{C++11} STL-based implementation of the new algorithm.

\section{D\"orfler marking}
\label{section:doerfler}

\subsection{Setting}\label{section:setting}
Let $0 < \theta < 1$ and $\II := \{ 1, \dots, N \}$. Given a vector $x \in \R^N_\star := \set{x \in \R^N\backslash\{0\}}{x_j \ge 0 \text{ for all } j \in \II}$, an index set $\MM \subseteq \II$ satisfies the \emph{D\"orfler criterion}, if
\begin{align}\label{eq:doerfler}
 \theta \sum_{j \in \II} x_j \le \sum_{j \in \MM} x_j.
\end{align}
By $\#\MM$, we denote the number of elements in $\MM$.
Let $\Nmin := \min \set{\#\MM}{\MM \subseteq \II \text{ satisfies } \eqref{eq:doerfler}}$ denote the minimal number of indices which are required to satisfy the D\"orfler criterion~\eqref{eq:doerfler}. We note that the minimizing set is not unique in general, e.g., if $x_i = x_j$ for all $i, j \in \II$ and $0<\theta\le (N-1)/N$.

\begin{remark}
For $\theta=1$, the set $\MM \subseteq \II$ of minimal cardinality satisfying \eqref{eq:doerfler} is unique and given by $\MM := \{j \in \II \colon x_j > 0\}$.
Clearly, this set can be determined at linear costs.
\end{remark}

We say that an algorithm realizes the \emph{minimal D\"orfler marking}, if, for all $0 < \theta < 1$, for all $N \in \N$, and for all $x \in \R^N_\star$, the algorithm constructs a set $\MM \subseteq \II$, which satisfies~\eqref{eq:doerfler} with $\#\MM = N_{\rm min}$. 
We say that an algorithm realizes the \emph{quasi-minimal D\"orfler marking}, if, for all $0 < \theta < 1$, there exists a constant $C \ge 1$ such that, for all $N \in \N$ and for all $x \in \R^N_\star$, the algorithm constructs a set $\MM \subseteq \II$, which satisfies~\eqref{eq:doerfler} with $\#\MM \le C \, N_{\rm min}$.

For current proofs of rate optimality of AFEM, the marking algorithm has to realize the quasi-minimal D\"orfler marking~\cite{cfpp2014}, while available results on optimal computational costs require also that the marking step has linear costs~\cite{stevenson2007,ghps2018,fhps2019}.

\subsection{Minimal D\"orfler marking based on sorting}
It is already noted in~\cite{doerfler1996} that a set $\MM \subseteq \II$, which satisfies~\eqref{eq:doerfler} as well as $\#\MM = \Nmin$, can easily be constructed by sorting.

\begin{algorithm}\label{algorithm:doerfler}
For the setting from Section~\ref{section:setting}, perform the following steps {\rm(i)--(iii)}:
\begin{itemize}
\item[\rm(i)] Determine a permutation $\pi: \II \to \II$ such that $x_{\pi(1)} \ge x_{\pi(2)} \ge \dots \ge x_{\pi(N)}$.
\item[\rm(ii)] Compute $v := \theta \, \sum_{j = 1}^N x_j$.
\item[\rm(iii)] Determine the minimal index $n \in \{1, \dots, N\}$ such that $v \le \sum_{i = 1}^n x_{\pi(i)}$.
\end{itemize}
\textbf{Output:} $\MM := \{\pi(1), \dots, \pi(n)\}$
\end{algorithm}

In practice, step~(i) of Algorithm~\ref{algorithm:doerfler} will be performed by sorting the vector $x \in \R^N_\star$.
This leads to $\OO(N \log N)$ operations for, e.g., the Introsort algorithm~\cite{musser1997}.

\begin{proposition}
The set $\MM$ generated by Algorithm~\ref{algorithm:doerfler} satisfies~\eqref{eq:doerfler} as well as $\#\MM = \Nmin$, i.e., Algorithm~\ref{algorithm:doerfler} realizes the minimal D\"orfler marking.
Up to step~{\rm(i)}, the computational cost of Algorithm~\ref{algorithm:doerfler} is linear.
\end{proposition}

\begin{proof}
Let $\MM_{\rm min} \subseteq \II$ satisfy~\eqref{eq:doerfler} with $\#\MM_{\rm min} = N_{\rm min}$. By construction of $\MM = \{ x_{\pi(1)}, \dots, x_{\pi(n)} \}$, it holds that
\begin{align*}
 \sum_{i=1}^{n-1} x_{\pi(i)} 
 < v 
 \le \sum_{j \in \MM_{\rm min}} x_j
 \le \sum_{i=1}^{n} x_{\pi(i)} \,. 
\end{align*}
Hence, we see that $n-1 < \#\MM_{\rm min} = N_{\rm min} \le n$. This implies that $n = N_{\rm min}$. It is obvious that step~(ii)--(iii) of Algorithm~\ref{algorithm:doerfler} have linear cost $\OO(N)$.
\end{proof}

\subsection{D\"orfler marking without sorting}

To avoid sorting, the work~\cite{doerfler1996} proposes (a variant of) the following algorithm; see~\cite[Section~5.2]{doerfler1996}.

\begin{algorithm}\label{algorithm:doerfler3}
For the setting from Section~\ref{section:setting} and given $0 < \nu <1$, perform the following steps {\rm(i)--(vi)}:
\begin{itemize}
\item[\rm(i)] Initialize $n := 0$, and $\pi(j) := 0$ for all $j = 1,\dots,N$.
\item[\rm(ii)] Compute $v := \theta \, \sum_{j = 1}^N x_j$ and $M := \max_{i = 1, \dots, N} x_i$.
\item[\rm(iii)] For $k =1, 2, 3,\dots, \lceil 1/\nu \rceil$, iterate the following steps:
\begin{itemize}
 \item[\rm(iv)] For all $i = 1, \dots, N$ with $i \not\in \set{\pi(j)}{j=1,\dots,n}$, iterate the following steps:
 \begin{itemize}
  \item[\rm(v)] if $x_i > (1-k\nu) M$, then define $\pi(n+1) := i$ and update $n \mapsto n+1$
  \item[\rm(vi)] if $v \le \sum_{j = 1}^n x_{\pi(j)}$, then terminate.
 \end{itemize}
\end{itemize}
\end{itemize}
\textbf{Output:} $\MM := \{\pi(1), \dots, \pi(n)\}$
\end{algorithm}

\begin{remark}
The algorithm proposed in \cite[Section~5.2]{doerfler1996} has the stopping criterion~{\rm(vi)} as part of step~{\rm(iii)}, i.e., steps~{\rm(iv)--(v)} are iterated, until $v\le\sum_{j=1}^N x_{\pi(j)}$.
If $x$ is constant, i.e., $x_j=c>0$ for all $j\in\II$, then this variant leads to $\MM=\II$ for all $0<\theta\le 1$ and hence does not realize quasi-minimal D\"orfler marking.
Our formulation of Algorithm~\ref{algorithm:doerfler3} excludes such a simple counterexample.
\end{remark}

\begin{proposition}\label{prop:doerfler3}
Algorithm~\ref{algorithm:doerfler3} terminates after finitely many steps. The computational cost of Algorithm~\ref{algorithm:doerfler3} is $\OO(N/\nu)$. The set $\MM$ generated by Algorithm~\ref{algorithm:doerfler3} realizes~\eqref{eq:doerfler}, but it is not quasi-minimal in general.
\end{proposition}

\begin{proof}
Steps~(i)--(ii) have linear costs $\OO(N)$.
Obviously, if in step~(vi) the sum is rather updated than recomputed, step~(iii)--(vi) lead to total costs $\OO(N/\nu)$ for Algorithm~\ref{algorithm:doerfler3}.
To see that $\MM$ satisfies~\eqref{eq:doerfler}, note that (at latest) for $k =  \lceil 1 / \nu \rceil$, it holds that $k\nu \ge 1$ and hence $x_i > (1 - k\nu) M$ is satisfied for all $x_i \neq 0$.
It only remains to show that Algorithm~\ref{algorithm:doerfler3} does not realize the quasi-minimal D\"orfler marking.

Let $0 < \theta < 1$ and $0 < \nu < 1$ be arbitrary.
We aim to show that for any constant $C\ge 1$, there exist $N\in\N$ and $x \in \R^N_\star$ such that the set $\MM$ generated by Algorithm~\ref{algorithm:doerfler3} satisfies $\#\MM > C N_{\rm min}$.
Without loss of generality, we may assume $C\in\N$.
The idea now is the following:
\begin{itemize}
  \item For some $R\in\N$ and $\varepsilon, \delta>0$, define the vector $x \in \R^N_\star$ of the form
  \begin{align*}
  x = (1,\,\, \underbrace{\varepsilon, \dots, \varepsilon}_{CR\text{ times}},\,\, \underbrace{\delta, \dots, \delta}_{R-1\text{ times}})\in\R^N\,, \quad\text{i.e.,}\quad
  x_j :=
  \begin{cases}
  1 &\text{if } j=1\,, \\
  \varepsilon &\text{if } 2 \le j \le N-R+1\,,\\
  \delta &\text{if } N-R+2 \le j \le N\,.
  \end{cases}
  \end{align*}
  \item Then, choose $0<\varepsilon\ll\delta\ll1$ and $R\in\N$ such that $\MM'=\{1\} \cup \{N-R+2, N-R+3,\dots, N\}$ satisfies \eqref{eq:doerfler}, but neither $\MM''=\{1\}$ nor $\MM''=\{1, \dots, CR+1\}$ do.
  \item If moreover $\delta$ and $\varepsilon$ are chosen such that the condition $x_i > (1-k\nu)M$ in Step~(v) of Algorithm~\ref{algorithm:doerfler3} is not satisfied for any of the indices $i=2,\dots, N$ and any of the loop iterations $k=1, \dots, \lceil 1/\nu \rceil -1$ of Step~(iii), then for the last loop iteration $k=\lceil 1/\nu\rceil$, starting from the index $i=2$, all indices $i=2, 3, \dots$ will be added to $\MM$ until \eqref{eq:doerfler} is satisfied.
  \item Now, if $\varepsilon>0$ is chosen small enough, then the set $\MM$ returned by Algorithm~\ref{algorithm:doerfler3} will be a superset of $\{1, \dots, CR+1\}$, i.e., $\#\MM > CR$.
  \item Since $\MM'=\{1\} \cup \{N-R+2, N-R+3, \dots, N\}$ satisfies \eqref{eq:doerfler}, it holds that $N_{\rm min} \le \#\MM' = R$ and hence $\#\MM > CR \ge CN_{\rm min}$.
\end{itemize}
It remains to define $\delta, \varepsilon$, and $R$ such that the desired properties hold.
Define
\begin{align*}
\delta &:= \left\lceil\left(1 - \nu\left(\left\lceil 1/\nu\right\rceil -1\right)\right)^{-1}\right\rceil^{-1}\,,\quad R := \left\lceil(2-\theta)/\theta\right\rceil/\delta + 1\,,\quad N := (C+1)R\,, \\
\varepsilon &:= (CR)^{-1}\min\left\{1, (1-\theta)\left(1+\left\lceil(2-\theta)/\theta\right\rceil\right)/\theta\right\} \,;
\end{align*}
Note that $1/\nu > \lceil 1/\nu\rceil -1$ implies that $\delta>0$.
First, note that
\begin{align*}
\theta\sum_{j=1}^N x_j 
&= \theta + \theta CR\,\varepsilon + \theta (R-1)\,\delta
= \theta + \theta CR\,\varepsilon + \theta\left\lceil(2-\theta)/\theta\right\rceil \\
&\le \theta + (1-\theta)\left(1+\left\lceil(2-\theta)/\theta\right\rceil\right) + \theta\left\lceil(2-\theta)/\theta\right\rceil \\
&= 1 + \left\lceil(2-\theta)/\theta\right\rceil = 1 + (R-1)\delta = \sum_{j\in\MM'} x_j \,.
\end{align*}
Hence, $\MM'$ satisfies \eqref{eq:doerfler} and therefore $N_{\rm min} \le \#\MM' = R$.
Next, we claim that Algorithm~\ref{algorithm:doerfler3} will construct a set $\MM\supsetneqq\{1, \dots, CR+1\}$, which thus contains more than $CR$ indices:
Observe that
\begin{align*}
0 &< \varepsilon \le (CR)^{-1} \le C^{-1}\delta\left\lceil(2-\theta)/\theta\right\rceil^{-1} \le C^{-1}\delta\theta/(2-\theta) < \delta \\
&= \left\lceil\left(1 - \nu\left(\left\lceil 1 / \nu\right\rceil -1\right)\right)^{-1}\right\rceil^{-1} \le 1 - \nu(\lceil 1/\nu\rceil - 1) \,.
\end{align*}
This proves that $0 < \varepsilon < \delta \le 1 - \nu(\lceil(1/\nu\rceil-1)$. 
Together with $M=x_1=1$, this implies that the condition $x_i > (1-k\nu)M$ in Step~(v) of Algorithm~\ref{algorithm:doerfler3} will not be satisfied for any $i \ge 2$ before the last iteration of the loop in Step~(iii) of Algorithm~\ref{algorithm:doerfler3} (i.e., before $k = \lceil 1/\nu\rceil$).
Thus, for $k < \lceil 1/\nu\rceil$, we have $\pi(1) = 1$, $n=1$, and $\pi(j) = 0$ for all $j = 2, \dots, N$.
Note, that 
\begin{align*}
\theta\sum_{j=1}^N x_j &= \theta + \theta CR\,\varepsilon + \theta (R-1)\,\delta \\
&> \theta + \theta (R-1)\,\delta = \theta + \theta\left\lceil(2-\theta)/\theta\right\rceil \ge 2 \ge 1 + CR\,\varepsilon = \sum_{j=1}^{CR+1} x_j \,.
\end{align*}
Consequently, after the last iteration of the $k$-loop it holds that $\pi(j) = j$ for all $j=1, \dots, CR +2$ and $n \ge CR+2$.
Hence, the set $\MM$ returned by Algorithm~\ref{algorithm:doerfler3} satisfies $\#\MM = n \ge CR+2 > CR$.
This concludes the proof.
\end{proof}

\subsection{Quasi-minimal D\"orfler marking with linear complexity by binning}

The following strategy has been proposed in the seminal work~\cite{stevenson2007}, which gave the first optimality proof for a standard AFEM loop of type~\eqref{eq:semr} for the 2D Poisson problem.
The main observation is the following:
If the reduction of the threshold in step~(v) of Algorithm~\ref{algorithm:doerfler3} is done by multiplication instead of subtraction, then the resulting algorithm satisfies the quasi-minimal D\"orfler marking.
While~\cite[Section~5]{stevenson2007} outlines the proposed strategy for the choice $\nu=1/2$, we work out all details in our proof of Proposition~\ref{prop:stevenson}.

\begin{algorithm}\label{algorithm:stevenson}
For the setting from Section~\ref{section:setting} and given $0 < \nu <1$, perform the following steps {\rm(i)--(v)}:
\begin{itemize}
\item[\rm(i)] Compute $v := \theta \sum_{i = 1}^N x_j$ and $M := \max_{j = 1, \dots, N} x_j$.
\item[\rm(ii)] Determine the minimal $K \in \N_0$ with $\nu^{K+1}M \le \frac{1-\theta}{\theta} \, v/N$. 
\item[\rm(iii)] For $k = 0, \dots, K$, fill bins $\BB_k := \set{j \in \II}{\nu^{k+1} < x_j/M \le \nu^k}$ and define $\BB_{K+1} := \II \backslash \bigcup_{k=0}^{K} \BB_k$.
\item[\rm(iv)] This yields a permutation $\pi: \II \to \II$ such that 
\begin{itemize}
\item[$\bullet$] $x_{\pi(i)} > x_{\pi(j)}$ for all $i \in \BB_I$ and $j \in \BB_J$ with $I < J$.
\end{itemize}
\item[\rm(v)] Determine the minimal index $n \in \{1, \dots, N\}$ such that $v \le \sum_{i = 1}^n x_{\pi(i)}$.
\end{itemize}
\textbf{Output:} $\MM := \{\pi(1), \dots, \pi(n)\}$
\end{algorithm}

\begin{proposition}\label{prop:stevenson}
For arbitrary $0 < \nu < 1$, Algorithm~\ref{algorithm:stevenson} terminates after finitely many steps. The constructed set $\MM \subseteq \II$ satisfies~\eqref{eq:doerfler} with $\#\MM \le \lceil\nu^{-1} N_{\rm min}\rceil$.
Moreover, a proper implementation of Algorithm~\ref{algorithm:stevenson} leads to a total computational cost of $\OO\big(N + K\big)$ with $K = \OO\big(\log_{1/\nu}(N/(1-\theta))\big)$.
\end{proposition}

\begin{proof}
The only non-obvious statement is the bound $\#\MM \le \lceil\nu^{-1} N_{\rm min}\rceil$:
For $j \in \BB_{K+1}$, it holds that $x_j \le \nu^{K+1}M \le \frac{1-\theta}{\theta} \, v/N$ and hence
\begin{align*}
  v + \frac{N}{\#\BB_{K+1}}\sum_{j\in\BB_{K+1}} x_j \le \frac{v}{\theta} = \sum_{j\in\II\setminus\BB_{K+1}} x_j + \sum_{j\in\BB_{K+1}} x_j \,.
\end{align*}
Since $\#\BB_{K+1} \le N$ and $x_j \ge 0$ for all $j \in \II$, it follows that $\bigcup_{k=0}^{k_0}\BB_k=\II\setminus\BB_{K+1}$ satisfies~\eqref{eq:doerfler}.
Let $k_0 \in \N_0$ be the largest index such that $\BB_{k_0} \subseteq \MM$.
If no such index exists, i.e., $\BB_0 \supsetneqq \MM$, define $k_0 := -1$.
Clearly, it holds that $k_0 \le K$ and $\bigcup_{k=0}^{k_0} \BB_k \subseteq \MM \subseteq \bigcup_{k=0}^{k_0+1} \BB_k$.
Further, there exists $\SS \subseteq \BB_{k_0+1}$ such that $\SS \cup \bigcup_{k=0}^{k_0}\BB_k$ satisfies \eqref{eq:doerfler} with minimal cardinality $N_{\rm min}$.

To show $\#\MM \le \lceil\nu^{-1} N_{\rm min}\rceil$, it suffices to show that $\RR := \MM \cap \BB_{k_0+1}$ satisfies $\#\RR \le \lceil\nu^{-1} \#\SS\rceil$.
Consider $\#\RR > 0$.
Then, $k_0 < K$, $\pi(n) \in \RR$, and with $v' := v - \sum_{k=0}^{k_0} \sum_{j\in\BB_k} x_j$, it holds that
\begin{align*}
\nu^{k_0+2}M(\#\RR -1) < \sum_{j\in\RR\setminus\{\pi(n)\}} x_j < v' \le \sum_{j\in\SS} x_j \le \nu^{k_0+1}M\#\SS \,.
\end{align*}
It immediately follows, that $\#\RR \le \lceil\nu^{-1}\#\SS\rceil$.
Altogether, $\MM$ satisfies~\eqref{eq:doerfler} with $\#\MM \le \lceil\nu^{-1} N_{\rm min}\rceil$.
\end{proof}

\section{Minimal D\"orfler marking with linear complexity}\label{section:main}
This section constitutes the main contribution of this work.
\begin{theorem}\label{thm:mainTheorem}
D\"orfler marking with minimal cardinality is a linear complexity problem.
More precisely, a call of Algorithm~\ref{algorithm:callQuickMark} below with a vector $x \in \R_\star^N$ leads after $\mathcal{O}(N)$ operations to a set $\MM \subseteq \{1, \dots, N\}$ with~\eqref{eq:doerfler} and $\#\MM = N_{\rm min}$.
\end{theorem}
We prove this main theorem explicitly by introducing the \alg{QuickMark} algorithm in Section~\ref{section:quickMark}.
The correctness of the \alg{QuickMark} algorithm is proved in Section~\ref{section:analysis} and the linear complexity of \alg{QuickMark} is shown in Section~\ref{section:complexity}.
Section~\ref{section:implementation} concludes with some remarks on the implementation of the algorithm.

\subsection{The \alg{QuickMark} algorithm}\label{section:quickMark}

Adapting the divide-and-conquer strategy of efficient selection algorithms~\cite{hoare1961_find}, we propose a new strategy to determine, at linear costs, a subset $\MM\subseteq\{1, \dots, N\}$ with \eqref{eq:doerfler} and $\#\MM = N_{\rm min}$.
The proposed algorithm consists of an initial call (Algorithm~\ref{algorithm:callQuickMark}) and the function \alg{QuickMark}~(Algorithm~\ref{algorithm:quickMark}), which steers the divide-and-conquer strategy based on the subroutines \alg{Pivot}~(Algorithm~\ref{algorithm:pivot}) and \alg{Partition}~(Algorithm~\ref{algorithm:partition}).

To improve readability throughout this chapter, whenever a permutation $\pi$ on $\{1, \dots N\}$ would be altered by a function, that function instead is written to take the permutation as input $\pi_\old$ and returns as output the new permutation $\pi_\new$.
If a permutation is not changed by a function, it is simply denoted by $\pi$.
Moreover, let $\pi_\id$ represent the identity permutation on $\{1, \dots, N\}$, i.e.,  $\pi_\id(j) = j$ for all $j \in \{1, \dots, N\}$.
For an index set $\JJ\subseteq\{1, \dots, N\}$ define $\pi(\JJ) := \{\pi(j)\colon j\in\JJ\}$.

\begin{algorithm}[Initial call of \alg{QuickMark}]{\label{algorithm:callQuickMark}}
For the setting from Section~\ref{section:setting}, we perform the following steps {\rm (i)--(iv)}:

\begin{itemize}
\item[\rm(i)] Initialize the identity permutation $\pi_\old := \pi_\id$.
\item[\rm(ii)] Define lower index $\ell := 1$ and upper index $u := N$.
\item[\rm(iii)] Compute the goal value $v := \theta\sum_{j=1}^N x_j$.
\item[\rm(iv)] Call $[\pi_\new, n] :=$~\alg{QuickMark}$(x, \pi_\old, \ell, u, v)$
\end{itemize}

\textbf{Output:} $\MM := \pi_\new(\{1, \dots, n\})$
\end{algorithm}

Analogously to selection algorithms~\cite{hoare1961_find}, the \alg{QuickMark} algorithm is based on the subroutine \alg{Partition}, where elements are essentially separated into two classes:
Those elements with smaller value than the pivot element, and those with greater value than the pivot element.
Then, the algorithm decides, which of the two classes is not to be inspected further.

\begin{algorithm}[{{$[\pi_\new, n] =$ \alg{QuickMark}$(x, \pi_\old, \ell, u, v)$}}]\label{algorithm:quickMark}
\textbf{Input:} Vector $x\in\R^N$, permutation $\pi_\old$ on $\{1, \dots, N\}$, goal value $v\in\R_{>0}$, lower and upper indices $1\le \ell\le u\le N$.
\begin{itemize}
\item[\rm(i)] Determine a pivot index $[p] :=$ \alg{Pivot}$(x, \pi_\old, \ell, u)$.
\item[\rm(ii)] Determine a new permutation via $[\pi_\new, g, s] :=$ \alg{Partition}$(x, \pi_\old, \ell, u, p)$.
\item[\rm(iii)] Compute the sum of the greatest elements $\sigma_g := \sum_{j=\ell}^g x_{\pi_\new(j)}$.
\item[\rm(iv)] If $\sigma_g \ge v$, then return \alg{QuickMark}$(x, \pi_\new, \ell, g, v)$
\item[\rm(v)] Else, if $\sigma_g + (s-g-1)x_{\pi_\old(p)} \ge v$, then return $[\pi_\new, g + \lceil(v-\sigma_g)/x_{\pi_\old(p)}\rceil]$
\item[\rm(vi)] Else return \alg{QuickMark}$(x, \pi_\new, s, u, v-\sigma_g - (s-g-1)x_{\pi_\old(p)})$
\end{itemize}
\textbf{Output:} Permutation $\pi_\new$ of $\{1, \dots, N\}$ and index $n \in \{1, \dots, N\}$.
\end{algorithm}

The \alg{Pivot} subroutine should determine a feasible pivot element of a given (sub-) array.
While the concrete choice of the pivot strategy is irrelevant for the correctness of the procedure, it is the decisive factor for the computational complexity of the divide-and-conquer strategy.
For now, we consider an arbitrarily (e.g., randomly) chosen $p \in \{\ell, \dots, u\}$.
While in Section~\ref{section:analysis} correctness of the algorithm is proved independently of the concrete pivot strategy, in Section~\ref{section:complexity} we propose a pivot strategy that leads --- even in the worst case --- to linear complexity $\mathcal{O}(N)$ of Algorithm~\ref{algorithm:callQuickMark}.

\begin{algorithm}[{{$[p] =$ \alg{Pivot}$(x, \pi, \ell, u)$}}]{\label{algorithm:pivot}}
\textbf{Input:} Vector $x\in\R^N$, permutation $\pi$ on $\{1, \dots, N\}$, lower and upper indices $1\le \ell\le u\le N$.
\begin{itemize}
\item[\rm(i)] Use $x_{\pi(\ell)}, x_{\pi(\ell+1)}, \dots, x_{\pi(u)}$ to determine a pivot index $p\in \{\ell, \dots, u\}$.
\end{itemize}
\textbf{Output:} Pivot index $p\in\{\ell, \dots, u\}$.
\end{algorithm}

For a given pivot element, the \alg{Partition} subroutine reorganizes the elements of an (sub-) array depending on whether they are greater than, smaller than, or equal to the pivot.

\begin{algorithm}[{{$[\pi_\new, g, s]$ = \alg{Partition}$(x, \pi_\old, \ell, u, p)$}}]{\label{algorithm:partition}}
\textbf{Input:} Vector $x\in\R^N$, permutation $\pi_\old$ on $\{1, \dots, N\}$, lower and upper indices $1\le \ell\le u\le N$, pivot index $\ell \le p \le u$.
\begin{itemize}
\item[\rm(i)] Compute a permutation $\pi_\mod$ on $\{\ell, \dots, u\}$ together with the unique indices $g \in \{\ell-1, \dots, u-1\}$ and $s \in \{\ell+1, \dots, u+1\}$ such that the following three implications hold true for all $j \in \{\ell, \dots, u\}$:
  \begin{itemize}
    \item[$\bullet$] If $x_{\pi_\old(\pi_\mod(j))} > x_{\pi_\old(p)}$, then $\ell \le j \le g$.
    \item[$\bullet$] If $x_{\pi_\old(\pi_\mod(j))} = x_{\pi_\old(p)}$, then $g < j < s$.
    \item[$\bullet$] If $x_{\pi_\old(\pi_\mod(j))} < x_{\pi_\old(p)}$, then $s \le j \le u$.
  \end{itemize}
  \item[\rm(ii)] Define $\pi_\new(j) := \begin{cases} \pi_\old(\pi_\mod(j)) \quad& \text{for } j \in \{\ell, \dots, u\},\\ \pi_\old(j) & \text{else}.\end{cases}$
\end{itemize}
\textbf{Output:} Permutation $\pi_\new$ of $\{1, \dots, N\}$ together with indices $g \in \{\ell-1, \dots, u - 1\}$ and $s \in \{\ell + 1, \dots, u+1\}$.
\end{algorithm}%

The following remark collects some important observations~\eqref{eq:prePartition:lu}--\eqref{eq:partiallyOrdered:gs} about the state of $\pi_\old$ and $\pi_\new$ in Algorithm~\ref{algorithm:partition}.
The validity of \eqref{eq:prePartition:lu} will be shown in Proposition~\ref{prop:wellPosedness} in Section~\ref{section:analysis}.
The properties \eqref{eq:partiallyOrdered:gs} follow directly from Algorithm~\ref{algorithm:partition}.
\begin{remark}
When \alg{Partition} (Algorithm~\ref{algorithm:partition}) is called in step~{\rm(ii)} of \alg{QuickMark} (Algorithm~\ref{algorithm:quickMark}), the permutation $\pi_\old$ and the indices $\ell, u$ satisfy
\begin{subequations}\label{eq:prePartition:lu}
\begin{alignat}{3}
\label{eq:prePartition:l} x_{\pi_\old(j)} &> x_{\pi_\old(k)} &&\quad\text{for all}\quad &&1\le j<\ell\le k\le N\,,\\
\label{eq:prePartition:u} x_{\pi_\old(j)} &> x_{\pi_\old(k)} &&\quad\text{for all}\quad &&1\le j \le u < k\le N\,.
\end{alignat}
\end{subequations}
This is illustrated in Figure~\ref{figure:prePartition}.

The permutation $\pi_\new$ defined in step~{\rm(ii)} of Algorithm~\ref{algorithm:partition} differs from $\pi_\old$ only at the indices $j\in\{\ell, \dots, u\}\subseteq\{1, \dots, N\}$.
Consequently, \eqref{eq:prePartition:l}--\eqref{eq:prePartition:u} are preserved by $\pi_\new$.
With the indices $g, s$ returned by Algorithm~\ref{algorithm:partition} and $p$ the pivot index, it additionally holds that
\begin{subequations}\label{eq:partiallyOrdered:gs}
\begin{alignat}{3}
\label{eq:partiallyOrdered:g} x_{\pi_\new(j)} &> x_{\pi_\old(p)} &\quad\text{for all}\quad &\ell\le j\le g\,,\\
\label{eq:partiallyOrdered:p} x_{\pi_\new(j)} &= x_{\pi_\old(p)} &\quad\text{for all}\quad &g < j < s\,,\\
\label{eq:partiallyOrdered:s} x_{\pi_\new(j)} &< x_{\pi_\old(p)} &\quad\text{for all}\quad &s \le j \le u\,.
\end{alignat}
\end{subequations}
This is illustrated in Figure~\ref{figure:postPartition}.
\end{remark}
\begin{figure}
\begin{tikzpicture}[scale=1.5]

\coordinate[label=-90:$1$] (left) at (0,0);
\coordinate[label=-90:$N$] (right) at (8,0);
\coordinate[] (up) at (0,0.43);
\coordinate[] (microUp) at (0,0.045);

\coordinate[label=135:$x\circ\pi_\old$] (name) at (left);

\coordinate[label=-90:$\ell$] (l) at ($(left)!0.2!(right)$);
\coordinate[label=-90:$u$] (u) at ($(left)!0.8!(right)$);
\coordinate[label=90:$\gg$] (tmp) at ($0.5*(microUp)+(left)!0.1!(right)$);
\coordinate[label=90:$\ll$] (tmp) at ($0.5*(microUp)+(left)!0.9!(right)$);

\draw[fill] (left) circle (.75pt);
\draw[fill] (right) circle (.75pt);
\draw[dashed, very thick, color=black!60!white] ($(l)+(microUp)$) -- ($(l)+(up)$);
\draw[fill] (l) circle (.75pt);
\draw[dashed, very thick, color=black!60!white] ($(u)+(microUp)$) -- ($(u)+(up)$);
\draw[fill] (u) circle (.75pt);

\draw[thick] (left)--(right)--($(right)+(up)$)--($(left)+(up)$)--cycle;

\end{tikzpicture}%
\caption{Ordering of $x\circ\pi_\old$ when \alg{Partition} is called, cf.~\eqref{eq:prePartition:lu}.
The array $x\circ\pi_\old$ is partially sorted in descending order:
The $\ell-1$ strictly largest values in $x\circ\pi_\old$ are obtained by the indices $\{1,\dots,\ell-1\}$. The $N-u$ strictly smallest values in $x\circ\pi_\old$ are obtained by the indices $\{u+1,\dots,N\}$.}
\label{figure:prePartition}
\end{figure}
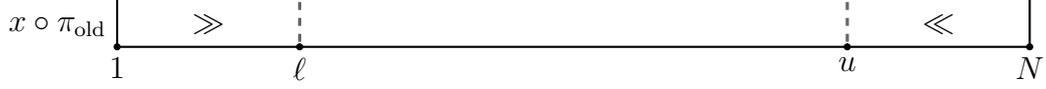
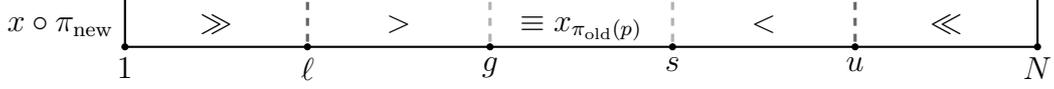
\begin{figure}
\begin{tikzpicture}[scale=1.5]

\coordinate[label=-90:$1$] (left) at (0,0);
\coordinate[label=-90:$N$] (right) at (8,0);
\coordinate[] (up) at (0,0.43);
\coordinate[] (microUp) at (0,0.045);
\coordinate[] (microDown) at (0,-0.045);

\coordinate[label=135:$x\circ\pi_\new$] (name) at (left);

\coordinate[label=-90:$\ell$] (l) at ($(left)!0.2!(right)$);
\coordinate[label=-90:$u$] (u) at ($(left)!0.8!(right)$);
\coordinate[label=90:$\gg$] (tmp) at ($0.5*(microUp)+(left)!0.1!(right)$);
\coordinate[label=90:$\ll$] (tmp) at ($0.5*(microUp)+(left)!0.9!(right)$);

\coordinate[label=-90:$g$] (g) at ($(left)!0.4!(right)$);
\coordinate[label=-90:$s$] (s) at ($(left)!0.6!(right)$);
\coordinate[label=90:$>$] (tmp) at ($0.5*(microUp)+(left)!0.3!(right)$);
\coordinate[label=90:$<$] (tmp) at ($0.5*(microUp)+(left)!0.7!(right)$);
\coordinate[label=90:$\equiv x_{\pi_\old(p)}$] (tmp) at ($(microDown)+(left)!0.5!(right)$);

\draw[fill] (left) circle (.75pt);
\draw[fill] (right) circle (.75pt);
\draw[dashed, very thick, color=black!60!white] ($(l)+(microUp)$) -- ($(l)+(up)$);
\draw[fill] (l) circle (.75pt);
\draw[dashed, very thick, color=black!60!white] ($(u)+(microUp)$) -- ($(u)+(up)$);
\draw[fill] (u) circle (.75pt);
\draw[dashed, very thick, color=black!30!white] ($(g)+(microUp)$) -- ($(g)+(up)$);
\draw[fill] (g) circle (.75pt);
\draw[dashed, very thick, color=black!30!white] ($(s)+(microUp)$) -- ($(s)+(up)$);
\draw[fill] (s) circle (.75pt);

\draw[thick] (left)--(right)--($(right)+(up)$)--($(left)+(up)$)--cycle;

\end{tikzpicture}%
\caption{Ordering of $x\circ\pi_\new$ when \alg{Partition} terminates, cf.~\eqref{eq:partiallyOrdered:gs}.
The properties~\eqref{eq:prePartition:lu} of $\pi_\old$ illustrated in Figure~\ref{figure:prePartition} are preserved.
Additionally, the reordered array $x\circ\pi_\new$ is partially sorted for the indices $\{\ell,\dots,u\}$: Within the index range $\{\ell,\dots,u\}$, the $g-\ell+1$ strictly largest values in $x\circ\pi_\new$ are obtained by the indices $\{\ell,\dots,g\}$, while the $u-s+1$ strictly smallest values in $x\circ\pi_\new$ are obtained by the indices $\{s,\dots,u\}$.
In particular, all indices $p'$ with $g<p'<s$ satisfy $x_{\pi_\new(p')} = x_{\pi_\old(p)}$ with the pivot-index $p$.}
\label{figure:postPartition}
\end{figure}

\subsection{Correctness of the \alg{QuickMark} algorithm}\label{section:analysis}
We consider $x\in\R_\star^N$, permutations $\pi$ on $\{1, \dots, N\}$, indices $\ell, u \in \{1, \dots, N\}$ with $1\le \ell \le u \le N$, and a value $v \in \R_{>0}$.
Proving the correctness of \alg{QuickMark} (Algorithm~\ref{algorithm:quickMark}) is organized into three steps:
In Section~\ref{section:admissibility} we verify some essential properties satisfied by the input parameters of calls to Algorithm~\ref{algorithm:quickMark}.
Section~\ref{section:MMM} introduces auxiliary subproblems generated and solved by Algorithm~\ref{algorithm:quickMark} and gives insight on the idea behind the \alg{QuickMark} strategy.
Termination of Algorithm~\ref{algorithm:quickMark} is investigated in Section~\ref{section:termination}, where the correctness is proved.

\subsubsection{{\bf Admissible calls to \alg{QuickMark}}}\label{section:admissibility}
We consider the following crucial properties, which will be shown to be always satisfied in Proposition~\ref{prop:wellPosedness}.
\begin{definition}\label{def:admissible}
A call \alg{QuickMark}$(x, \pi_\old, \ell, u, v)$ to Algorithm~\ref{algorithm:quickMark} is called \textbf{admissible}, if the inputs $x\in\R_\star^N, \pi_\old, \ell, u, v$ satisfy the following conditions {\rm(a)}--{\rm(b)}:
\begin{itemize}
  \item[\rm(a)] It holds that
\begin{subequations}\label{eq:partiallyOrdered:lu}
\begin{alignat}{3}
  \label{eq:partiallyOrdered:l} x_{\pi_\old(j)} &> x_{\pi_\old(k)} &\quad\text{for all}\quad &1\le j < \ell \le k \le N\,,\\
  \label{eq:partiallyOrdered:u} x_{\pi_\old(j)} &> x_{\pi_\old(k)} &\quad\text{for all}\quad &1\le j \le u < k \le N\,.
\end{alignat}
\end{subequations}
  \item[\rm(b)] It holds that
\begin{align}\label{eq:admissibleV}
0 < v = \theta\sum_{j=1}^N x_j - \sum_{j=1}^{\ell-1} x_{\pi_\old(j)} \le \sum_{j=\ell}^u x_{\pi_\old(j)}\,.
\end{align}
\end{itemize}
\end{definition}

In fact, the following proposition shows that recursive calls of \alg{QuickMark} preserve the admissibility conditions.
\begin{proposition}\label{prop:wellPosedness}
If \alg{QuickMark} is initially called by Algorithm~\ref{algorithm:callQuickMark}{\rm(iv)}, then each subsequent recursive call \alg{QuickMark}$(x, \pi, \ell, u, v)$ from step {\rm(iv)} or {\rm(vi)} of Algorithm~\ref{algorithm:quickMark} is admissible.
\end{proposition}
\begin{proof}
The statement follows directly by induction.
First, we show that the initial call \alg{QuickMark}$(x, \pi_\old, \ell, u, v)$ of Algorithm~\ref{algorithm:quickMark} initiated by Algorithm~\ref{algorithm:callQuickMark}{\rm(iv)} with the inputs $x\in\R_\star^N$, $\pi_\old := \pi_\id$, $\ell := 1$, $u := N$, and $v := \theta\sum_{j=1}^N x_j$ is admissible:
Since $\ell=1$ and $u=N$, Definition~\ref{def:admissible}{\rm(a)} contains only statements about indices in the empty set and is therefore satisfied.
Definition~\ref{def:admissible}{\rm(b)} follows from $x\in\R_\star^N$, $0 < \theta < 1$, and the definition of $v$.

For the induction step, consider an admissible call \alg{QuickMark}$(x, \pi_\old, \ell, u, v)$ of Algorithm~\ref{algorithm:quickMark}.
We show that a potential subsequent call \alg{QuickMark}$(x, \pi_\new, \ell', u', v')$ initiated by either Algorithm~\ref{algorithm:quickMark}{\rm(iv)} (i.e., $\ell'=\ell$, $u'=g$, $v'=v$), or by Algorithm~\ref{algorithm:quickMark}{\rm(vi)} (i.e., $\ell'=s$, $u'=u$, $v'=v-\sum_{j=\ell}^{s-1} x_{\pi_\new(j)}$), is also admissible:
By {\rm(a)}, {\rm(b)} we refer to the assumption, i.e., the admissibility conditions of Definition~\ref{def:admissible} satisfied by \alg{QuickMark}$(x, \pi_\old, \ell, u, v)$.
We aim to show the corresponding admissibility conditions of Definition~\ref{def:admissible} for the call \alg{QuickMark}$(x, \pi_\new, \ell', u', v')$, which will be denoted by $\rm(a')$, $\rm(b')$.

Recall, that in either case (step~{\rm(iv)} or step~{\rm(vi)} in Algorithm~\ref{algorithm:quickMark}), $\pi_\new$ differs from $\pi_\old$ only on the index set $\{\ell, \dots, u\}\subseteq\{1, \dots, N\}$.
Therefore, in both cases $\rm(a')$ follows from \eqref{eq:partiallyOrdered:g}--\eqref{eq:partiallyOrdered:s} and {\rm(a)}.
If recursion relies on Algorithm~\ref{algorithm:quickMark}{\rm(iv)}, then $\ell'=\ell$, $u'=g$, and $v':=v\le\sigma_g$.
Hence,
\begin{align*}
  0 < v' =v \stackrel{{\rm(b)}}{=} \theta\sum_{j=1}^N x_j - \sum_{j=1}^{\ell-1} x_{\pi_\old(j)} = \theta\sum_{j=1}^N x_j - \sum_{j=1}^{\ell'-1} x_{\pi_\new(j)} \le \sigma_g = \sum_{j=\ell'}^{u'} x_{\pi_\new(j)}
\end{align*}
proves $\rm(b')$.
If recursion relies on Algorithm~\ref{algorithm:quickMark}{\rm(vi)}, then $\ell'=s$, $u'=u$, and
\begin{align}\label{eq:auxProofVI}
  v > \sigma_g + (s-g-1)x_{\pi_\old(p)} \stackrel{\eqref{eq:partiallyOrdered:p}}{=} \sum_{j=\ell}^{s-1} x_{\pi_\new(j)} \,.
\end{align}
Combining {\rm(b)} and the last estimate yields for $v' := v - \sum_{j=\ell}^{s-1} x_{\pi_\new(j)}$ that
\begin{align*}
  0 \stackrel{\eqref{eq:auxProofVI}}{<}  v' \stackrel{{\rm(b)}}{=} \theta\sum_{j=1}^N x_j - \sum_{j=1}^{\ell-1} x_{\pi_\old(j)} - \sum_{j=\ell}^{s-1} x_{\pi_\new(j)} = \theta\sum_{j=1}^N x_j - \sum_{j=1}^{\ell'-1} x_{\pi_\new(j)} \stackrel{{\rm(b)}}{\le} \sum_{j=\ell'}^{u'} x_{\pi_\new(j)}\,.
\end{align*}
This shows $\rm(b')$.
\end{proof}

\subsubsection{{\bf Subproblems generated by \alg{QuickMark}}}\label{section:MMM}
To analyze Algorithm~\ref{algorithm:quickMark}, we introduce some auxiliary notation.
In particular, the symbol $\MM$ will be used differently than in Section~\ref{section:setting}.
The connection between the two notations is clarified in Remark~\ref{remark:MMnotation}.

By $\PP(\{\ell, \dots, u\})$, we denote the power set of $\{\ell, \dots, u\}$.
For any admissible call \alg{QuickMark}$(x, \pi, \ell, u, v)$ to Algorithm~\ref{algorithm:quickMark}, let $\MMM(x, \pi, \ell, u, v) \subseteq \PP(\{\ell, \dots, u\})$ consist of all $\MM \in \PP(\{\ell, \dots, u\})$ such that
\begin{subequations}\label{equation:MMM}
\begin{alignat}{2}
  \label{equation:MMM:a} x_{\pi(j)} &\ge x_{\pi(k)} &\quad &\text{for all } j\in\MM  \text{ and all } k\in\{\ell, \dots, u\}\setminus\MM,\\
  \label{equation:MMM:b} \sum_{j\in\MM} x_{\pi(j)} &\ge v > \sum_{j\in\MM\setminus\{k\}} x_{\pi(j)} &\quad &\text{for all } k\in\MM\,.
\end{alignat}
\end{subequations}
The following remark follows immediately from \eqref{equation:MMM:a}--\eqref{equation:MMM:b} and connects the introduced notation to the D\"orfler marking criterion~\eqref{eq:doerfler} from Section~\ref{section:setting}.
\begin{remark}\label{remark:MMnotation}
For arbitrary $\MM \in \MMM(x, \pi, 1, N, \theta\sum_{j=1}^N x_j)$, the set $\MM' := \pi(\MM) \in \MMM(x, \pi_\id, 1, N, \theta\sum_{j=1}^N x_j)$ satisfies \eqref{eq:doerfler} with minimal cardinality $\#\MM' = N_{\rm min}$.
\end{remark}

Later in Section~\ref{section:termination}, we will prove that \alg{QuickMark} called by Algorithm~\ref{algorithm:callQuickMark} determines a set $\MM\in\MMM(x, \pi_\id, 1, N, \theta\sum_{j=1}^N x_j)$.
The core idea behind the proof is the observation that for an admissible call \alg{QuickMark}$(x, \pi, \ell, u, v)$, the set $\MMM(x, \pi_\id, 1, N, \theta\sum_{j=1}^N x_j)$ can be written as
\begin{align*}
\Big\{\pi(\{1, \dots, \ell-1\}) \cup \pi(\MM') \colon \MM' \in \MMM(x, \pi, \ell, u, v)\Big\} \,.
\end{align*}
Hence, an admissible call \alg{QuickMark}$(x, \pi_\old, \ell, u, v)$ to Algorithm~\ref{algorithm:quickMark} either determines a set $\MM\in\MMM(x, \pi_\new, \ell, u, v)$ and terminates in step~{\rm(v)}, or it initiates another admissible recursive call denoted by \alg{QuickMark}$(x, \pi_\new, \ell', u', v')$ in step~{\rm(iv)} or step~{\rm(vi)}, where $\{\ell', \dots, u'\} \subsetneqq \{\ell, \dots, u\}$, i.e., the problem is reduced to a strict subproblem.

First, we will show, that all occurring subproblems of finding $\MM\in\MMM(x, \pi, \ell, u, v)$ are well-posed.
In fact, for an admissible call \alg{QuickMark}$(x, \pi, \ell, u, v)$ the set $\MMM(x, \pi, \ell, u, v)$ is always nonempty and all $\MM\in\MMM(x, \pi, \ell, u, v)$ attain the same minimum in $x\circ\pi$.
\begin{lemma}\label{lemma:xStarUnique}
Let \alg{QuickMark}$(x, \pi, \ell, u, v)$ be an admissible call to Algorithm~\ref{algorithm:quickMark}.
Then, $\MMM(x, \pi, \ell, u, v) \not= \emptyset$. 
Moreover, the definition
\begin{align}\label{definition:xstar}
x^\ast(x, \pi, \ell, u, v) := \min_{j\in\MM} x_{\pi(j)}
\end{align}
is independent of the concrete choice of $\MM\in\MMM(x, \pi, \ell, u, v)$.
\end{lemma}

\begin{proof}
To show that $\MMM(x, \pi, \ell, u, v) \not= \emptyset$, we explicitly construct some $\MM\in\MMM(x, \pi, \ell, u, v)$:
Starting with $\MM_0 := \{\ell, \dots, u\}$, for $i=0, \dots, u-\ell$ define
\begin{align*}
  m_i := \min\{j\in\MM_i\colon x_{\pi(j)}=\min_{k\in\MM_i} x_{\pi(k)}\} \quad\text{and}\quad
  \MM_{i+1} := \MM_i\setminus\{m_i\}\,,
\end{align*}
i.e., $\MM_{i+1}$ is generated by extracting the index with the smallest value in $x\circ\pi$ from $\MM_i$.
By construction, \eqref{equation:MMM:a} holds for all $\MM_i$, $i=0,\dots, u-\ell+1$.
Further, the values $\sum_{j\in\MM_i} x_{\pi(j)}$ are monotonically decreasing in $i=0,\dots, u-\ell+1$.
Since $\MM_{u-\ell+1} = \emptyset$, the admissibility~\eqref{eq:admissibleV} of $v$ implies that
\begin{align*}
  \sum_{j\in\MM_{u-\ell+1}} x_{\pi(j)} = 0 < v \le \sum_{j=\ell}^u x_{\pi(j)} = \sum_{j\in\MM_0} x_{\pi(j)}\,.
\end{align*}
Consequently, there exists a unique $i'\in\{0, \dots, u-\ell\}$ such that 
\begin{align*}
  \sum_{j\in\MM_{i'+1}} x_{\pi(j)} < v \le \sum_{j\in\MM_{i'}} x_{\pi(j)}\,.
\end{align*}
By construction, for all $i=0,\dots,u-\ell$ (and in particular for $i=i'$) it holds that
\begin{align*}
\sum_{j\in\MM_{i}\setminus\{k\}} x_{\pi(j)} \le -m_i + \sum_{j\in\MM_i} x_{\pi(j)} = \sum_{j\in\MM_{i+1}} x_{\pi(j)} \quad\text{for all } k\in\MM_i\,.
\end{align*}
Hence, combining the last two estimates shows that $\MM_{i'}$ also satisfies \eqref{equation:MMM:b} and thus $\MM_{i'}\in\MMM(x, \pi, \ell, u, v)$. This proves $\MMM(x, \pi, \ell, u, v) \not= \emptyset$.

To show that the definition \eqref{definition:xstar} is independent of $\MM\in\MMM(x, \pi, \ell, u, v)$, we claim that
\begin{align*}
x_1^\ast := \min_{j\in\MM_1}x_{\pi(j)} = \min_{j\in\MM_2}x_{\pi(j)} =: x_2^\ast \quad \text{for all} \quad \MM_1, \MM_2 \in \MMM(x, \pi, \ell, u, v) \,.
\end{align*}
To prove this claim, we argue by contradiction and assume $x_1^\ast \not= x_2^\ast$ and, without loss of generality, $x_1^\ast < x_2^\ast$.
Hence, we have $\MM_1 \setminus \MM_2 \not= \emptyset$ and
\begin{align*}
x_1^\ast < x_2^\ast \le x_{\pi(k)} \quad\text{for all } k\in\MM_2\,.
\end{align*}
If there exists $k\in\MM_2\setminus\MM_1$, then \eqref{equation:MMM:a} gives that $x_1^\ast \ge x_{\pi(k)}$.
This contradicts the last estimate and hence proves that $\MM_2\setminus\MM_1=\emptyset$.
Therefore, we deduce that $\MM_2 \subsetneqq \MM_1$.
Using the second inequality in \eqref{equation:MMM:b} for $\MM_1$ and then using the first inequality in \eqref{equation:MMM:b} for $\MM_2$ , we see that
\begin{align*}
v > -x_1^\ast + \sum_{j\in\MM_1} x_{\pi(j)} \ge \sum_{j\in\MM_2} x_{\pi(j)} \ge v\,.
\end{align*}
This contradiction implies that $x_1^\ast = x_2^\ast$ and concludes the proof.
\end{proof}

\subsubsection{{\bf Termination of \alg{QuickMark}}}\label{section:termination}
For any admissible call \alg{QuickMark}$(x, \pi_\old, \ell, u, v)$ of Algorithm~\ref{algorithm:quickMark}, exactly one of three cases --- recursion by step~{\rm(iv)}, termination by step~{\rm(v)}, or recursion by step~{\rm(vi)} --- applies.
The next lemma connects the termination in step~{\rm(v)} directly to the pivot index chosen in step~{\rm(i)}.

\begin{lemma}\label{lemma:xStarTermination}
Let \alg{QuickMark}$(x, \pi_\old, \ell, u, v)$ be an admissible call to Algorithm~\ref{algorithm:quickMark}.
Then, Algorithm~\ref{algorithm:quickMark} terminates with step~{\rm(v)}, if and only if the pivot index $p\in\{\ell,\dots, u\}$ from step~{\rm(i)} satisfies $x_{\pi_\old(p)} = x^\ast(x, \pi_\old, \ell, u, v)$.
\end{lemma}

\begin{proof}
After step~{\rm(ii)} of Algorithm~\ref{algorithm:quickMark}, it holds that $\pi_\new(\{\ell, \dots, u\}) = \pi_\old(\{\ell, \dots, u\})$ and hence $x^\ast(x, \pi_\new, \ell, u, v) = x^\ast(x, \pi_\old, \ell, u, v)$. 

First, suppose that Algorithm~\ref{algorithm:quickMark} terminates with step~{\rm(v)}, i.e.,
\begin{align*}
  \sum_{j=\ell}^g x_{\pi_\new(j)} = \sigma_g < v \le \sigma_g + (s-g-1)x_{\pi_\old(p)} \stackrel{\eqref{eq:partiallyOrdered:p}}{=} \sum_{j=\ell}^{s-1} x_{\pi_\new(j)} \,.
\end{align*}
Now \eqref{eq:partiallyOrdered:g}--\eqref{eq:partiallyOrdered:s} imply that
\begin{align*}
\{\ell, \dots, g\}\subsetneqq\MM\subseteq\{\ell, \dots, s-1\} \quad\text{for all } \MM\in\MMM(x, \pi_\new, \ell, u, v)\,.
\end{align*}
By definition~\eqref{definition:xstar} and \eqref{eq:partiallyOrdered:g}--\eqref{eq:partiallyOrdered:p}, it follows that $x_{\pi_\old(p)} = x^\ast(x, \pi_\new, \ell, u, v)$.

Conversely, suppose that $x_{\pi_\old(p)} = x^\ast(x, \pi_\new, \ell, u, v)$ and let $\MM\in\MMM(x, \pi_\new, \ell, u, v)$ be arbitrary.
Then, \eqref{eq:partiallyOrdered:g}--\eqref{eq:partiallyOrdered:s} and $x_{\pi_\old(p)} = \min_{j\in\MM} x_{\pi_\new(j)}$ imply that
\begin{align*}
\{\ell, \dots, g\}\subsetneqq\MM\subseteq\{\ell, \dots, s-1\}\,.
\end{align*}
Therefore, \eqref{equation:MMM:b} leads to
\begin{align*}
  \sigma_g = \sum_{j=\ell}^g x_{\pi_\new(j)} < v \le \sum_{j=\ell}^{s-1} x_{\pi_\new(j)} \stackrel{\eqref{eq:partiallyOrdered:p}}{=} \sigma_g + (s-g-1)x_{\pi_\old(p)} \,.
\end{align*}
Consequently, Algorithm~\ref{algorithm:quickMark} terminates in step~{\rm(v)}.
\end{proof}

Whenever an admissible call of Algorithm~\ref{algorithm:quickMark} terminates in step~{\rm(v)}, a solution to the corresponding auxiliary subproblem is provided.

\begin{lemma}\label{lemma:localDoerfler}
Let \alg{QuickMark}$(x, \pi_\old, \ell, u, v)$ be an admissible call to Algorithm~\ref{algorithm:quickMark}.
If \alg{QuickMark}$(x, \pi_\old, \ell, u, v)$ terminates in step~{\rm(v)}, then the output $[\pi_\new, n]$ guarantees that $\MM:=\{\ell, \dots, n\}\in\MMM(x, \pi_\new, \ell, u, v)$.
\end{lemma}
\begin{proof}
With $p, \pi_\new, g, s, \sigma_g$ from steps~{\rm(i)}--{\rm(iii)}, the termination in Algorithm~\ref{algorithm:quickMark}{\rm(v)} implies that 
\begin{align*}
  \sigma_g < v \le \sigma_g + (s-g-1)x_{\pi_\old(p)}\,.
\end{align*}
Obviously, $x_{\pi_\old(p)} > 0$. Together with \eqref{eq:partiallyOrdered:gs}, this shows that $n := g + \lceil (v-\sigma_g)/x_{\pi_\old(p)}\rceil$ returned in Algorithm~\ref{algorithm:quickMark}{\rm(v)} satisfies that $g < n < s$.
Again, \eqref{eq:partiallyOrdered:gs} implies that $\MM=\{\ell, \dots, n\}$ satisfies \eqref{equation:MMM:a}.
It remains to show \eqref{equation:MMM:b}: The definition of $\sigma_g:=\sum_{j=\ell}^{g} x_{\pi_\new(j)}$ and the choice of $n$ show that for all $k\in\MM$ it holds
\begin{align*}
&\quad \sum_{j\in\MM\setminus\{k\}} x_{\pi_\new(j)} \le -x_{\pi_\old(p)} + \sum_{j\in\MM} x_{\pi_\new(j)} \stackrel{\eqref{eq:partiallyOrdered:p}}{=} \sum_{j=\ell}^{n-1} x_{\pi_\new(j)} = \sigma_g + \sum_{j=g+1}^{n-1} x_{\pi_\new(j)} \\
&\stackrel{\eqref{eq:partiallyOrdered:p}}{=} \sigma_g + (n-g-1)\, x_{\pi_\old(p)} = \sigma_g + (\lceil (v-\sigma_g)/x_{\pi_\old(p)}\rceil - 1)\, x_{\pi_\old(p)} < \sigma_g + v - \sigma_g = v\,.
\end{align*}
Similarly, we see that 
\begin{align*}
v &= v - \sigma_g + \sigma_g  \le \lceil (v-\sigma_g)/x_{\pi_\old(p)}\rceil x_{\pi_\old(p)} + \sigma_g\\
&= (n-g) x_{\pi_\old(p)} + \sigma_g \stackrel{\eqref{eq:partiallyOrdered:p}}{=} \sum_{j=g+1}^{n} x_{\pi_\new(j)} + \sigma_g = \sum_{j\in\MM} x_{\pi_\new(j)}\,.
\end{align*}
Consequently, $\MM$ satisfies \eqref{equation:MMM:b} and we conclude that $\MM := \{\ell, \dots, n\} \in \MMM(x,\pi_\new, \ell, u, v)$.
\end{proof}

Algorithm~\ref{algorithm:callQuickMark} always terminates and provides a set of minimal cardinality satisfying the D\"orfler marking criterion.

\begin{theorem}
If initially called by Algorithm~\ref{algorithm:callQuickMark}, then \alg{QuickMark} terminates after finitely many operations and the output $[\pi_\new, n]$ guarantees that $\pi_\new(\{1, \dots, n\})$ satisfies the D\"orfler criterion~\eqref{eq:doerfler} with minimal cardinality.
\end{theorem}

\begin{proof}
At latest the $(N-1)$-st recursive call of \alg{QuickMark} terminates in step~{\rm(v)} of Algorithm~\ref{algorithm:quickMark}:
Proposition~\ref{prop:wellPosedness} shows that all (subsequent) calls of \alg{QuickMark} are admissible.
For any recursive call \alg{QuickMark}$(x, \pi_\new, \ell', u', v')$ initiated by step~{\rm(iv)} or step~{\rm(vi)} of \alg{QuickMark}$(x, \pi_\old, \ell, u, v)$, it holds that $\{\ell', \dots, u'\}\subsetneqq\{\ell, \dots, u\}$.
Therefore, if none of the first $N-2$ recursive calls of \alg{QuickMark} terminates in step~{\rm(v)} of Algorithm~\ref{algorithm:quickMark}, for the $(N-1)$-st recursive call denoted by \alg{QuickMark}$(x, \bar{\pi}, \bar{\ell}, \bar{u}, \bar{v})$ it holds that $\bar{\ell}=\bar{u}$.
Consequently, for this call the pivot index is chosen as $\bar{p}=\bar{\ell}=\bar{u}$ in step~{\rm(i)} of Algorithm~\ref{algorithm:quickMark}.
Using Lemma~\ref{lemma:xStarUnique}, the admissibility of \alg{QuickMark}$(x, \bar{\pi}, \bar{\ell}, \bar{u}, \bar{v})$ implies that $\MMM(x, \bar{\pi}, \bar{\ell}, \bar{u}, \bar{v}) \not= \emptyset$.
We infer that $\{\bar{p}\} \in \MMM(x, \bar{\pi}, \bar{\ell}, \bar{u}, \bar{v})$ and thus
\begin{align*}
x^\ast(x, \bar{\pi}, \bar{\ell}, \bar{u}, \bar{v}) 
\stackrel{\eqref{definition:xstar}}{=} \min_{j\in\{\bar{p}\}} x_{\bar{\pi}(j)}
= x_{\bar{\pi}(\bar{p})}\,.
\end{align*}
Hence, Lemma~\ref{lemma:xStarTermination} implies termination of \alg{QuickMark}$(x, \bar{\pi}, \bar{\ell}, \bar{u}, \bar{v})$ in Algorithm~\ref{algorithm:quickMark}{\rm(v)}.

It remains to show that $\MM':=\pi_\new(\{1, \dots, n\})$ satisfies \eqref{eq:doerfler} with minimal cardinality.
In view of Remark~\ref{remark:MMnotation}, we will show that $\MM:=\{1, \dots, n\}\in\MMM(x, \pi_\new, 1, N, \theta\sum_{j=1}^N x_j)$:
Suppose that $[\pi_\new, n]$ are obtained by Algorithm~\ref{algorithm:callQuickMark}{\rm(iv)}.
Denote the last recursive call of Algorithm~\ref{algorithm:quickMark} by \alg{QuickMark}$(x, \bar{\pi}_\old, \bar{\ell}, \bar{u}, \bar{v})$.
By Proposition~\ref{prop:wellPosedness}, this call is admissible and $\pi_\new (= \bar{\pi}_\new)$ differs from $\bar{\pi}_\old$ only for the indices $\{\bar{\ell}, \dots, \bar{u}\}\subseteq\{1, \dots, N\}$.

By Lemma~\ref{lemma:localDoerfler}, it holds that $\{\bar{\ell}, \dots, n\}\in\MMM(x, \pi_\new, \bar{\ell}, \bar{u}, \bar{v})$.
Thus, the partial ordering \eqref{eq:partiallyOrdered:l}--\eqref{eq:partiallyOrdered:u} shows that
\begin{align}\label{eq:MMsat_a}
x_{\bar{\pi}_\new(j)} \ge x_{\bar{\pi}_\new(k)} \quad \text{for all } j\in\MM  \text{ and all } k\in\{1, \dots, N\}\setminus\MM\,.
\end{align}
By Definition~\ref{def:admissible}{\rm(b)}, it holds that
\begin{align*}
  \bar{v} = \theta\sum_{j=1}^N x_j - \sum_{j=1}^{\bar{\ell}-1} x_{\bar{\pi}_\old(j)} = \theta\sum_{j=1}^N x_j - \sum_{j=1}^{\bar{\ell}-1} x_{\pi_\new(j)} \,.
\end{align*}
Since $\{\bar{\ell}, \dots, n\}\in\MMM(x, \pi_\new, \bar{\ell}, \bar{u}, \bar{v})$, condition \eqref{equation:MMM:b} reads
\begin{align*}
\sum_{j=\bar{\ell}}^n x_{\pi_\new(j)} &\ge \bar{v} > -x_{\pi_\new(k)} + \sum_{j=\bar{\ell}}^n x_{\pi_\new(j)} &\quad &\text{for all } \bar{\ell}\le k\le n\,.
\end{align*}
Using the partial ordering \eqref{eq:partiallyOrdered:l} and adding $\sum_{j=1}^{\bar{\ell}-1} x_{\pi_\new(j)}$ to the last estimate, we get 
\begin{align}\label{eq:MMsat_b}
\sum_{j=1}^n x_{\pi_\new(j)} &\ge \theta\sum_{j=1}^N x_j > -x_{\pi_\new(k)} + \sum_{j=1}^n x_{\pi_\new(j)} &\quad &\text{for all } 1 \le k \le N\,.
\end{align}
Consequently, \eqref{eq:MMsat_a}--\eqref{eq:MMsat_b} show that $\MM\in\MMM(x, \pi_\new, 1, N, \theta\sum_{j=1}^N x_j)$.
\end{proof}

\subsection{Computational complexity of the \alg{QuickMark} algorithm}\label{section:complexity}

Exploiting the fact that selection problems can always be solved in linear time~\cite{BFPRT1973}, we show that the pivoting strategy in Algorithm~\ref{algorithm:pivot} can be chosen such that, for any $x\in\R_\star^N$ and any $0<\theta<1$, Algorithm~\ref{algorithm:callQuickMark} always terminates after $\OO(N)$ operations.
Consider choosing the median of $\{x_{\pi(j)}\colon j=\ell, \dots, u\}$ as pivot element.
\begin{algorithm}[{{$[p] =$ \alg{Median}$(x, \pi, \ell, u)$}}]{\label{algorithm:median}}
\textbf{Input:} Vector $x\in\R^N$, permutation $\pi$ on $\{1, \dots, N\}$, lower and upper index $1\le \ell\le u\le N$.
\begin{itemize}
\item[\rm(i)] Determine an index $p\in \{\ell, \dots, u\}$ such that
\begin{subequations}\label{eq:median}
\begin{align}
  \label{eq:median:smaller} \#\{j\in\{\ell,\dots,u\}\colon x_{\pi(j)} < x_{\pi(p)}\} &\le (u-\ell+1)/2\,,\\
  \label{eq:median:greater} \#\{j\in\{\ell,\dots,u\}\colon x_{\pi(j)} > x_{\pi(p)}\} &\le (u-\ell+1)/2\,.
\end{align}
\end{subequations}
\end{itemize}
\textbf{Output:} Median index $p$.
\end{algorithm}

According to~\cite{BFPRT1973}, Algorithm~\ref{algorithm:median} can be implemented such that it always terminates in linear time $\OO(u-\ell+1)$. This leads to the following theorem.
\begin{theorem}\label{thm:complexity}
If \alg{Pivot} is replaced by \alg{Median} in Algorithm~\ref{algorithm:quickMark}{\rm(i)}, then, for any $x\in\R_\star^N$ and any $0<\theta<1$, Algorithm~\ref{algorithm:callQuickMark} terminates after $\OO(N)$ operations.
In particular, the multiplicative constant hidden in the Landau notation is generic and independent of $\theta$ and $N$.
\end{theorem}

\begin{proof}
Obviously, steps~{\rm(i)}--{\rm(iii)} of Algorithm~\ref{algorithm:callQuickMark} can be realized using $\OO(N)$ operations.
Moreover, the permutation $\pi$ can be represented by additionally storing an array containing $N$ indices.
It remains to show that the call to \alg{QuickMark} in step~{\rm(iv)} terminates at linear costs $\OO(N)$.

Consider a (possibly recursive) call of \alg{QuickMark}$(x, \pi_\old, \ell, u, v)$. The median (-index) of $x\circ\pi$ with respect to the indices $\{\ell, \dots, u\}$ of Algorithm~\ref{algorithm:quickMark}{\rm(i)} can be determined at linear cost $\OO(u-\ell+1)$; see~\cite[Theorem~1]{BFPRT1973}.
The partition in Algorithm~\ref{algorithm:quickMark}{\rm(ii)} can be determined at linear cost $\OO(u-\ell+1)$.
In particular, this can easily be implemented by temporarily storing not more than $u-\ell+1$ additional indices $\pi_\mod$.
Algorithm~\ref{algorithm:quickMark}{\rm(iii)} is of cost $g-\ell+1 < u-\ell+1$ and steps~{\rm(iv)}--{\rm(vi)} of Algorithm~\ref{algorithm:quickMark} are of constant cost $\OO(1)$ plus, in the case of step~{\rm(iv)} or  step~{\rm(vi)}, the cost of the recursive call on at most $(u-\ell+1)/2$ indices; see \eqref{eq:median}.
We have shown that for a generic constant $C\ge 1$, the costs for an iteration of Algorithm~\ref{algorithm:quickMark} are bounded by $C(u-\ell+1)$ plus the costs of a potential recursive call.

Now, denote the computational costs of a call of \alg{QuickMark}$(x, \pi, \ell, u, v)$ by $T(m)$, where $m = \#\{\ell, \dots, u\} = u-\ell+1$  is the number of elements under consideration.
Then, due to the choice \alg{Pivot} $:=$ \alg{Median}, using \eqref{eq:median:greater} in Algorithm~\ref{algorithm:quickMark}{\rm(iv)}, or \eqref{eq:median:smaller} in Algorithm~\ref{algorithm:quickMark}{\rm(vi)}, respectively, it follows inductively that
\begin{align*}
T(N) \le CN + T(N/2) \le \dots \le CN\sum_{j=0}^\infty 2^{-j} = 2CN\,.
\end{align*}
For the choice \alg{Pivot} $:=$ \alg{Median}, we conclude that Algorithm~\ref{algorithm:quickMark}, and hence Algorithm~\ref{algorithm:callQuickMark}, always terminates at linear costs.
\end{proof}

\begin{remark}\label{re:pivoting}
{\rm(i)} In the complexity estimate of Theorem~\ref{thm:complexity} the dependency on $0<\theta<1$ is avoided due to the choice of \alg{Median} as pivoting strategy.
Other pivoting strategies may lead to a hidden constant depending on $0<\theta<1$.\\
{\rm(ii)} If Algorithm~\ref{algorithm:quickMark}{\rm(i)} chooses the pivot index $p \in \{\ell, \dots, u\}$ always randomly, then the algorithm might perform faster on average.
  However, this would lead to quadratic worst-case performance $\OO(N^2)$ of Algorithm~\ref{algorithm:callQuickMark}.\\
{\rm(iii)} Theorem~\ref{thm:complexity} is proved for choosing the $50\%$-quantile, i.e., the median element is the pivot (Algorithm~\ref{algorithm:median}).
  If any other fixed quantile is chosen as the pivot, then Theorem~\ref{thm:complexity} still holds true.\\
{\rm(iv)} If for fixed $q \in (0,1)$ one chooses pivoting by the $q$-quantile rather than by the median in Theorem~\ref{thm:complexity}, then a call of \alg{QuickMark}($x, \pi_\old, \ell, u, v$) potentially leads to a recursive call in step~{\rm(iv)} or step~{\rm(vi)} of Algorithm~\ref{algorithm:quickMark} on up to $\max\{q, 1-q\}(u-\ell+1)$ indices.
Hence, the computational costs of Algorithm~\ref{algorithm:quickMark} with this pivoting strategy called on $N$ indices can then be estimated by
\begin{align*}
T(N) \le CN + T(\max\{q, 1-q\}N) \le \dots \le CN\sum_{j=0}^\infty \max\{q, 1-q\}^j = \frac{CN}{\min\{q, 1-q\}}\,.
\end{align*}
Obviously, choosing the median as pivot (i.e., $q = 1/2$) optimizes this estimate.
\end{remark}

\subsection{Remarks on the implementation of \alg{QuickMark}}\label{section:implementation}

Up to now, we focused on the idea and the theoretical aspects of the \alg{QuickMark} algorithm, namely verifying Theorem~\ref{thm:mainTheorem}.
We conclude this section by discussing some adaptions to the algorithm as it is presented in Section~\ref{section:quickMark}, in order to arrive at an efficient competitive \texttt{C++11} implementation using routines provided by the standard library.
Ultimately, we compare the performance of our implementation to an implementation of Algorithm~\ref{algorithm:doerfler} based on the sorting routine provided by the standard library.

The following observations lead to an efficient \alg{QuickMark} implementation relying on routines provided by the standard library.
\begin{remark}\label{re:implementation}
{\rm(i)} The data structure for given refinement indicators $\eta_\ell(T)$ for all $T\in\TT_\ell$ is usually a vector \emph{\texttt{eta}}, where \emph{\texttt{eta[j]}} refers to the estimated error for the $j$-th element in the data structure representing the mesh $\TT_\ell$.
To preserve this relation, one aims to avoid manipulating (i.e., reordering) \emph{\texttt{eta}}.\\
{\rm(ii)} \alg{QuickMark} as formulated in Algorithm~\ref{algorithm:quickMark} avoids manipulation of \emph{\texttt{eta}} by operating on a permutation $\pi$ only.
Hence, in a straight forward implementation of Algorithm~\ref{algorithm:quickMark}, which uses a permutation $\pi$ to access elements of the array $x\circ\pi$, data is not accessed contiguously and a considerable performance penalty is introduced.\\
{\rm(iii)} Hence, to achieve a more efficient implementation of \alg{QuickMark}, one would rather alter the algorithm to operate on (and modify) a temporary copy \emph{\texttt{x}} of \emph{\texttt{eta}} to determine the value $x^\ast := x^\ast(\emph{\texttt{eta}}\,, \pi_\id, 1, N, \theta\sum_{j=1}^N x_j)$.
The desired set $\MM$ is then given by the union of $\{j\colon \emph{\texttt{eta[j]}} > x^\ast\}$ and a proper subset of $\{j\colon \emph{\texttt{eta[j]}} = x^\ast\}$.\\
{\rm(iv)} For the ease of presentation, in \alg{Partition}~(Algorithm~\ref{algorithm:partition}) a partition into three subarrays --- elements strictly greater than, equal to, and strictly smaller than the pivot element --- is demanded.
In view of using standard library partition implementations, we note that this is not necessary:
It suffices to partition into two subarrays:
One with elements greater than or equal to the pivot element, the pivot element itself, and one with elements smaller than or equal to the pivot element.
Then, as long as it is ensured, that other elements with the same value as the pivot element are distributed evenly among the two subarrays, Theorem~\ref{thm:complexity} holds true.\\
{\rm(v)} When using a partition based algorithm to determine a quantile, e.g., the median element, as the pivot element, the subarray is already partitioned after Algorithm~\ref{algorithm:quickMark}{\rm(i)}.
Hence, Algorithm~\ref{algorithm:quickMark}{\rm(ii)} can be skipped.
\end{remark}

\lstset{language=C++,
                basicstyle=\scriptsize\ttfamily,
                keywordstyle=\color{green!50!black}\ttfamily,
                commentstyle=\color{blue!50!black}\ttfamily,
                rulecolor=\color{black},
                frame=single
}

Using headers \texttt{<vector>}, \texttt{<iterator>}, \texttt{<algorithm>}, \texttt{<functional>}, and \texttt{<numeric>}, a \texttt{C++11} implementation of \alg{QuickMark} adapted to the observations of Remark~\ref{re:implementation} relying on routines from the standard library could read as follows.
\begin{lstlisting}
using Iterator_t = std::vector<double>::iterator;
const double xStarKernel
(Iterator_t subX_begin, Iterator_t subX_end, double goal)
{
  // QuickMark, step (i)-(ii): partition by median element
  auto length {std::distance(subX_begin, subX_end)};
  auto subX_middle {subX_begin + length/2};
  std::nth_element(subX_begin, subX_middle, subX_end, std::greater<double>());
  auto pivot_val {*subX_middle};

  // QuickMark, step (iii)
  auto sigma_g {std::accumulate(subX_begin, subX_middle, (double)0.0)};

  // QuickMark, step (iv), (v) and (vi)
  if (sigma_g >= goal)
    return xStarKernel(subX_begin, subX_middle, goal);
  if (sigma_g + pivot_val >= goal)
    return pivot_val;
  return xStarKernel(++subX_middle, subX_end, goal - sigma_g - pivot_val);
}
\end{lstlisting}

Passing refinement indicators $\eta_\ell(T)$ for all $T\in\TT_\ell$ (\texttt{eta}) and an adaptivity parameter $0<\theta<1$ (\texttt{theta}) to the following adaption of Algorithm~\ref{algorithm:callQuickMark}, then yields the desired value $x^\ast$, such that the set $\MM$ is readily obtained; see Remark~\ref{re:implementation}{\rm(iii)}.
\begin{lstlisting}
const double compute_xStar (const std::vector<double>& eta, double theta)
{
  std::vector<double> x {eta};
  double goal {theta * std::accumulate(x.cbegin(), x.cend(), (double)0.0)};
  return xStarKernel(x.begin(), x.end(), goal);
}
\end{lstlisting}

\begin{remark}
While \alg{QuickMark} can be implemented such that its complexity is linear even in the worst case, the worst-case complexity of the given \texttt{C++} function \texttt{xStarKernel} is (standard library-) implementation dependent:\\
The \texttt{C++} standard requires \texttt{std::nth\_element} to be of linear complexity only on average, while lacking any worst-case restriction~\cite{isoC++}.
A quality introspective selection implementation of \texttt{std::nth\_element} could be realized as proposed in~\cite{musser1997}:
As fast as the Quickselect algorithm~\cite{hoare1961_find} in practice, maintaining linear worst-case complexity by relying on the median of medians algorithm from~\cite{BFPRT1973} as fallback strategy.
\end{remark}

We conclude by comparing the performance of the \texttt{C++} standard library implementation \texttt{std::sort} to our implementation \texttt{xStarKernel} above.
This is reasonable, since those two routines are the core components of Algorithm~\ref{algorithm:doerfler} and Algorithm~\ref{algorithm:callQuickMark} (adapted to the observations of Remark~\ref{re:implementation}), respectively.
The completing components of Algorithm~\ref{algorithm:doerfler} and Algorithm~\ref{algorithm:callQuickMark} are very similar for both approaches and in particular, make up for only a small fraction of the overall computational cost of the respective algorithm.

We consider adaptivity parameters $\theta \in \{0.1, 0.25, 0.5, 0.75, 0.9\}$ and vectors of length $N \in \{10^j\colon j=3, \dots, 9\}$.
For each combination of $\theta$ and $N$ we generate $\num{30}$ vectors \texttt{eta} of length $N$ filled with uniformly distributed pseudorandom double-precision values between $0$ and $1$.
The core routines \texttt{std::sort} and \texttt{xStarKernel} are called on (copies of) each of these vectors and the computational times are measured.
The sources were compiled with GNU compiler \texttt{g++} version \texttt{5.5.0}, optimization flag \texttt{-O3}, and \texttt{-std=c++11} enabled.
All computations were performed on a machine with \SI{32}{\giga\byte} of RAM and an Intel Core i7-6700 CPU~\cite{i7-6700} with a base frequency of \SI{3.4}{\giga\hertz}.

For all test cases $(\theta, N) \in \{0.1, 0.25, 0.5, 0.75, 0.9\}\times\{10^j\colon j=3, \dots, 9\}$, the measured times for the fastest (Table~\ref{fig:table_min}), average (Table~\ref{fig:table_avg}) and slowest (Table~\ref{fig:table_max}) run out of $\num{30}$ runs is given.
To emphasize the improved complexity of Algorithm~\ref{algorithm:callQuickMark} over Algorithm~\ref{algorithm:doerfler}, the measurements for $\theta = 0.5$ are visualized in Figure~\ref{fig:complexity}:
While the computational time spent per element increases logarithmically with the problem size for \texttt{std::sort}, it remains constant for \texttt{xStarKernel}.
Hence, as expected, the \alg{QuickMark} strategy clearly outperforms the approach of Algorithm~\ref{algorithm:doerfler} based on sorting.
Moreover, while the measured time behaves like $\mathcal{O}(N\log N)$ for sorting, it only grows linearly with respect to the problem size for \alg{QuickMark} as predicted by Theorem~\ref{thm:complexity}.
In accordance with Theorem~\ref{thm:complexity}, different values of $0<\theta<1$ do not influence the performance of the algorithm.

\begin{table}
\footnotesize
\sisetup{
  output-exponent-marker = \text{e},
  table-format=+1.1e+1,
  exponent-product={},
  retain-explicit-plus,
  retain-zero-exponent,
  round-mode=places, 
  round-precision=1
}
 \begin{tabular}{||c||c|c||c|c||c|c||c|c||c|c||}
    \hline
    \multirow{2}{*}{$N$} & \multicolumn{2}{c||}{$\theta=0.1$} & \multicolumn{2}{c||}{$\theta=0.25$} & \multicolumn{2}{c||}{$\theta=0.5$} & \multicolumn{2}{c||}{$\theta=0.75$} & \multicolumn{2}{c||}{$\theta=0.9$} \\
    \cline{2-11}
   & \texttt{sort} & \texttt{xStar} & \texttt{sort} & \texttt{xStar}& \texttt{sort} & \texttt{xStar}& \texttt{sort} & \texttt{xStar}& \texttt{sort} & \texttt{xStar} \\
  \hline
$10^3$ & \num{3.42e-05} & \num{1.45e-05} & \num{3.46e-05} & \num{1.32e-05} & \num{3.47e-05} & \num{1.19e-05} & \num{3.41e-05} & \num{9.90e-06} & \num{3.41e-05} & \num{1.37e-05} \\ 
$10^4$ & \num{4.47e-04} & \num{1.31e-04} & \num{4.44e-04} & \num{1.20e-04} & \num{4.49e-04} & \num{1.12e-04} & \num{4.41e-04} & \num{8.48e-05} & \num{4.53e-04} & \num{1.14e-04} \\ 
$10^5$ & \num{5.46e-03} & \num{1.15e-03} & \num{5.41e-03} & \num{1.21e-03} & \num{5.50e-03} & \num{1.12e-03} & \num{5.47e-03} & \num{1.11e-03} & \num{5.49e-03} & \num{1.13e-03} \\ 
$10^6$ & \num{6.56e-02} & \num{1.20e-02} & \num{6.48e-02} & \num{1.14e-02} & \num{6.50e-02} & \num{1.21e-02} & \num{6.54e-02} & \num{1.12e-02} & \num{6.54e-02} & \num{1.12e-02} \\ 
$10^7$ & \num{7.60e-01} & \num{1.17e-01} & \num{7.59e-01} & \num{1.09e-01} & \num{7.56e-01} & \num{1.09e-01} & \num{7.63e-01} & \num{1.23e-01} & \num{7.59e-01} & \num{1.09e-01} \\ 
$10^8$ & \num{8.67e+00} & \num{1.21e+00} & \num{8.71e+00} & \num{1.10e+00} & \num{8.69e+00} & \num{1.21e+00} & \num{8.68e+00} & \num{1.07e+00} & \num{8.73e+00} & \num{1.13e+00} \\ 
$10^9$ & \num{9.79e+01} & \num{1.17e+01} & \num{9.75e+01} & \num{1.17e+01} & \num{9.79e+01} & \num{1.13e+01} & \num{9.75e+01} & \num{1.16e+01} & \num{9.81e+01} & \num{1.19e+01} \\ 
  \hline
\end{tabular}

\smallskip
\caption{Measured time (in seconds) for finding $x^\ast$ of a given double-precision vector of length $N$, versus the time it takes to sort it.
Times for the fastest run out of \num{30} runs.}
 \label{fig:table_min}
\end{table}


\begin{table}
\footnotesize
\sisetup{
  output-exponent-marker = \text{e},
  table-format=+1.1e+1,
  exponent-product={},
  retain-explicit-plus,
  retain-zero-exponent,
  round-mode=places, 
  round-precision=1
}
 \begin{tabular}{||c||c|c||c|c||c|c||c|c||c|c||}
    \hline
    \multirow{2}{*}{$N$} & \multicolumn{2}{c||}{$\theta=0.1$} & \multicolumn{2}{c||}{$\theta=0.25$} & \multicolumn{2}{c||}{$\theta=0.5$} & \multicolumn{2}{c||}{$\theta=0.75$} & \multicolumn{2}{c||}{$\theta=0.9$} \\
    \cline{2-11}
   & \texttt{sort} & \texttt{xStar} & \texttt{sort} & \texttt{xStar}& \texttt{sort} & \texttt{xStar}& \texttt{sort} & \texttt{xStar}& \texttt{sort} & \texttt{xStar} \\
  \hline
$10^3$ & \num{3.63e-05} & \num{1.65e-05} & \num{3.56e-05} & \num{1.56e-05} & \num{3.57e-05} & \num{1.57e-05} & \num{3.75e-05} & \num{1.62e-05} & \num{3.51e-05} & \num{1.51e-05} \\ 
$10^4$ & \num{4.61e-04} & \num{1.45e-04} & \num{4.64e-04} & \num{1.45e-04} & \num{4.64e-04} & \num{1.40e-04} & \num{4.67e-04} & \num{1.40e-04} & \num{4.78e-04} & \num{1.42e-04} \\ 
$10^5$ & \num{5.62e-03} & \num{1.38e-03} & \num{5.67e-03} & \num{1.40e-03} & \num{5.88e-03} & \num{1.41e-03} & \num{5.65e-03} & \num{1.36e-03} & \num{5.56e-03} & \num{1.34e-03} \\ 
$10^6$ & \num{6.66e-02} & \num{1.40e-02} & \num{6.65e-02} & \num{1.37e-02} & \num{6.65e-02} & \num{1.36e-02} & \num{6.64e-02} & \num{1.30e-02} & \num{6.64e-02} & \num{1.33e-02} \\ 
$10^7$ & \num{7.70e-01} & \num{1.40e-01} & \num{7.69e-01} & \num{1.38e-01} & \num{7.70e-01} & \num{1.34e-01} & \num{7.70e-01} & \num{1.35e-01} & \num{7.71e-01} & \num{1.35e-01} \\ 
$10^8$ & \num{8.77e+00} & \num{1.41e+00} & \num{8.77e+00} & \num{1.37e+00} & \num{8.76e+00} & \num{1.38e+00} & \num{8.78e+00} & \num{1.36e+00} & \num{8.79e+00} & \num{1.33e+00} \\ 
$10^9$ & \num{9.87e+01} & \num{1.39e+01} & \num{9.87e+01} & \num{1.34e+01} & \num{9.87e+01} & \num{1.34e+01} & \num{9.88e+01} & \num{1.30e+01} & \num{9.88e+01} & \num{1.34e+01} \\ 
  \hline
\end{tabular}

\smallskip
\caption{Measured time (in seconds) for finding $x^\ast$ of a given double-precision vector of length $N$, versus the time it takes to sort it.
Average time for a run out of \num{30} runs.}
 \label{fig:table_avg}
\end{table}


\begin{table}
\footnotesize
\sisetup{
  output-exponent-marker = \text{e},
  table-format=+1.1e+1,
  exponent-product={},
  retain-explicit-plus,
  retain-zero-exponent,
  round-mode=places, 
  round-precision=1
}
 \begin{tabular}{||c||c|c||c|c||c|c||c|c||c|c||}
    \hline
    \multirow{2}{*}{$N$} & \multicolumn{2}{c||}{$\theta=0.1$} & \multicolumn{2}{c||}{$\theta=0.25$} & \multicolumn{2}{c||}{$\theta=0.5$} & \multicolumn{2}{c||}{$\theta=0.75$} & \multicolumn{2}{c||}{$\theta=0.9$} \\
    \cline{2-11}
   & \texttt{sort} & \texttt{xStar} & \texttt{sort} & \texttt{xStar}& \texttt{sort} & \texttt{xStar}& \texttt{sort} & \texttt{xStar}& \texttt{sort} & \texttt{xStar} \\
  \hline
$10^3$ & \num{6.02e-05} & \num{1.83e-05} & \num{3.66e-05} & \num{1.75e-05} & \num{3.73e-05} & \num{1.77e-05} & \num{6.65e-05} & \num{2.87e-05} & \num{3.68e-05} & \num{1.74e-05} \\ 
$10^4$ & \num{4.98e-04} & \num{1.76e-04} & \num{5.12e-04} & \num{1.87e-04} & \num{5.13e-04} & \num{1.63e-04} & \num{5.04e-04} & \num{1.65e-04} & \num{5.19e-04} & \num{1.67e-04} \\ 
$10^5$ & \num{6.22e-03} & \num{1.62e-03} & \num{6.09e-03} & \num{1.61e-03} & \num{6.83e-03} & \num{1.88e-03} & \num{6.05e-03} & \num{1.94e-03} & \num{5.72e-03} & \num{1.55e-03} \\ 
$10^6$ & \num{6.84e-02} & \num{1.64e-02} & \num{6.84e-02} & \num{1.56e-02} & \num{6.80e-02} & \num{1.52e-02} & \num{6.79e-02} & \num{1.54e-02} & \num{6.77e-02} & \num{1.59e-02} \\ 
$10^7$ & \num{7.85e-01} & \num{1.56e-01} & \num{7.78e-01} & \num{1.54e-01} & \num{7.80e-01} & \num{1.54e-01} & \num{7.78e-01} & \num{1.54e-01} & \num{8.11e-01} & \num{1.56e-01} \\ 
$10^8$ & \num{8.85e+00} & \num{1.58e+00} & \num{8.84e+00} & \num{1.52e+00} & \num{8.85e+00} & \num{1.57e+00} & \num{8.84e+00} & \num{1.51e+00} & \num{8.90e+00} & \num{1.48e+00} \\ 
$10^9$ & \num{1.0e+02} & \num{1.59e+01} & \num{1.0e+02} & \num{1.51e+01} & \num{1.0e+02} & \num{1.55e+01} & \num{1.0e+02} & \num{1.53e+01} & \num{1.0e+02} & \num{1.50e+01} \\ 
  \hline
\end{tabular}

\smallskip
\caption{Measured time (in seconds) for finding $x^\ast$ of a given double-precision vector of length $N$, versus the time it takes to sort it.
Slowest time for a run out of \num{30} runs.}
 \label{fig:table_max}
\end{table}

\begin{figure}
\begin{tikzpicture}
\pgfplotstableread{plots/mark_theta0.50.dat}{\dataMark}
\pgfplotstableread{plots/sort_theta0.50.dat}{\dataSort}
\pgfplotsset{every tick label/.append style={font=\footnotesize}}
\begin{semilogxaxis}[
width = .65\textwidth,
height = 6cm,
xlabel={\footnotesize{Array length \footnotesize{$N$}}},
ylabel={\footnotesize{$T(N) / N$ ($\SI{}{\nano\second}$)}},
xtick={1e3, 1e5, 1e7, 1e9},
ytick={10, 30, 50, 70, 90},
legend pos=outer north east,
]
\addplot[purple, mark size=1,very thick] table[x=N, y expr={(\thisrow{max} / \thisrow{N}) * 1e9}] {\dataSort};
\addplot[magenta, mark size=1, very thick] table[x=N, y expr={(\thisrow{avg} / \thisrow{N}) * 1e9}] {\dataSort};
\addplot[orange, mark size=1, very thick] table[x=N, y expr={(\thisrow{min} / \thisrow{N}) * 1e9}] {\dataSort};
\addplot[black, mark size=1,very thick, dashed] table[x=N,y expr={(10*log10(\thisrow{N})}] {\dataMark};
\addplot[black, mark size=1,very thick, dash dot dot] table[x=N,y expr={(\thisrow{N} / \thisrow{N}) * 20}] {\dataMark};
\addplot[olive, mark size=1,very thick] table[x=N, y expr={(\thisrow{max} / \thisrow{N}) * 1e9}] {\dataMark};
\addplot[teal, mark size=1,very thick] table[x=N, y expr={(\thisrow{avg} / \thisrow{N}) * 1e9}] {\dataMark};
\addplot[green, mark size=1,very thick] table[x=N, y expr={(\thisrow{min} / \thisrow{N}) * 1e9}] {\dataMark};
\addlegendentry{$\!\!\!\!${\scriptsize\texttt{std::sort} slowest}}
\addlegendentry{$\!\!\!${\scriptsize\texttt{std::sort} average}}
\addlegendentry{$\!\!\!\!\!${\scriptsize\texttt{std::sort} fastest}}
\addlegendentry{{\scriptsize$T(N) = \Theta(N\log N)$}}
\addlegendentry{{\scriptsize$\!\!\!\!\!\!\!\!\!\!\!\!\!\!\!T(N) = \Theta(N)$}}
\addlegendentry{$\,${\scriptsize\texttt{xStarKernel} slowest}}
\addlegendentry{$\;${\scriptsize\texttt{xStarKernel} average}}
\addlegendentry{{\scriptsize\texttt{xStarKernel} fastest}}
\end{semilogxaxis}

\end{tikzpicture}
\caption{Visualization of the data from Tables~\ref{fig:table_min}--\ref{fig:table_max} for $\theta = 0.5$.
Performance comparison of \texttt{std::sort} versus \texttt{xStarKernel}:
Computational time spent per element $T(N) / N$ in nanoseconds versus the array length $N$.
}\label{fig:complexity}
\end{figure}
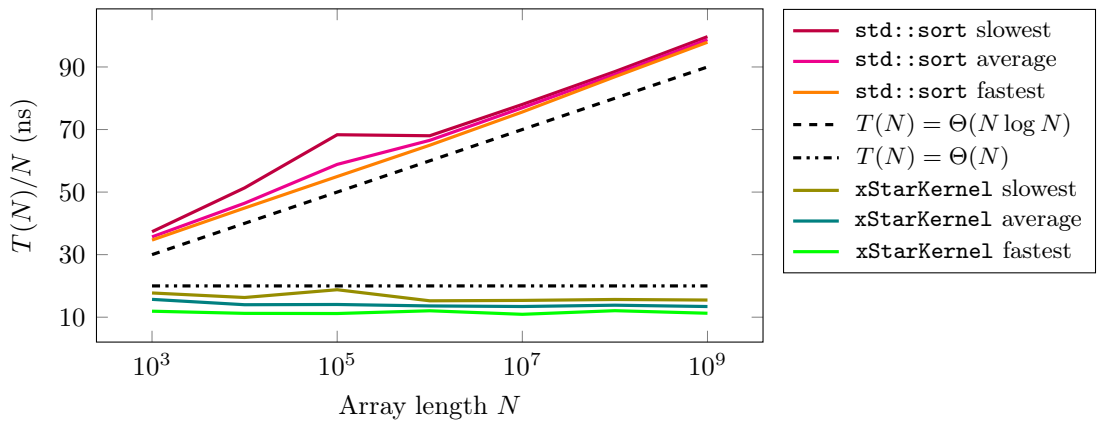


\bibliographystyle{alpha}
\bibliography{ref}

\end{document}